\documentclass[11pt]{amsart}

\newtheorem{tm}{Theorem}
\newtheorem{lemma}{Lemma}

\newtheorem{prop}{Proposition}

\usepackage{amsmath}
\usepackage[psamsfonts]{amssymb}
\usepackage[mathscr]{euscript}

\newcommand{\cal}{\mathcal}

\title{Generalized metrics and generalized twistor spaces}
\author{Johann Davidov}

\thanks{The author is partially supported by  the National Science
Fund, Ministry of Education and Science of Bulgaria under contract
DN 12/2.}

\address{Institute of Mathematics and Informatics \\
Bulgarian Academy of Sciences\\ Acad. G.Bonchev St. Bl.8\\
1113 Sofia\\ Bulgaria} \email{jtd@math.bas.bg}

\begin{document}

\begin{abstract}
The twistor construction for Riemannian manifolds is extended to the case of  manifolds endowed with generalized metrics
(in the sense of generalized geometry \`a la Hitchin). The generalized twistor space associated to such a manifold is
defined as the bundle of generalized complex structures on the tangent spaces of the manifold compatible with the given
generalized metric. This space admits natural generalized almost complex structures whose integrability conditions are
found in the paper. An interesting feature of the generalized twistor spaces discussed in it is the existence of
intrinsic isomorphisms.

 \vspace{0,1cm} \noindent 2010 {\it Mathematics Subject Classification} 53D18, 53C28.

\vspace{0,1cm} \noindent {\it Key words: generalized complex structures, twistor
spaces}

\end{abstract}

\maketitle \vspace{0.5cm}

\section{Introduction}

The concept of generalized complex geometry has been introduced by Nigel Hitchin \cite{Hit02} and further developed by
his students M. Gualtieri \cite{Gu}, G. Cavalcanti \cite{Ca}, F. Witt \cite{Witt} as well as by many other mathematicians
and physicists (including Hitchin himself). A  generalized almost complex structure in the sense of Hitchin \cite{Hit02}
on a smooth manifold $M$ is an endomorphism ${\mathcal J}$ of the bundle $TM\oplus T^{\ast}M$ satisfying ${\mathcal
J}^2=-Id$ and compatible with the metric $<X+\alpha,Y+\beta>=\alpha(Y)+\beta(X)$. Similar to the case of a usual almost
complex structure, the integrability condition for  a generalized almost complex structure ${\mathcal J}$ is defined as
the vanishing of its  Nijenhuis tensor. However, for ${\mathcal J}$ this tensor is  defined by means of the bracket,
introduced by T. Courant \cite{Cou}, instead of the Lie bracket. If ${\mathcal J}$ is integrable, it is called a
generalized complex structure. Every complex and every symplectic structure determines a generalized complex structure in
a natural way. There are several  examples of generalized complex structures which are not defined by means of a complex
or a symplectic structure, to quote just a few of them \cite{CaGu04,CG09,CaG11,Gu,Hit05}. In \cite{Bre,DM06,DM07,
Des,GS,P14} such examples have been given by means of the Penrose twistor construction \cite{Pen,PeWa} as developed by
Atiyah, Hitchin and Singer \cite{AtHiSi} in the framework of Riemannian geometry. While the base manifold of the twistor
space considered in \cite{DM06,DM07}  is not equipped with a metric, the base manifold in \cite{Des} is a
four-dimensional Riemannian manifold $M$ and the one in \cite{Bre,GS} is a hyper-K\"ahler manifold. The fiber of the
twistor space in \cite {Des} consists of (linear) generalized complex structures on the tangent spaces of the base
manifold compatible with the metric on $TM\oplus T^{\ast}M$ induced by the metric of $M$. This construction can be placed
and generalized in the framework of the concept of a generalized metric, introduced by Gualtieri \cite{Gu} and Witt
\cite{Witt}.

A generalized metric on a vector space $T$ is a subspace $E$ of $T\oplus T^{\ast}$ such that $dim\,E=dim\,T$ and  the
metric $<.\,,.>$ is positive definite on $E$. Every generalized metric is uniquely determined by a positive definite
metric $g$ and a skew-symmetric $2$-form $\Theta$ on $T$ so that $E=\{X+\imath_{X}g+\imath_{X}\Theta: X\in T\}$. It is
convenient to set $E'=E$ and $E''=E^{\perp}$, the orthogonal complement of $E$ with respect to $<.\,,.>$. Then $T\oplus
T^{\ast}=E'\oplus E''$ and the restrictions to $E'$ and $E''$ of the projection $pr_{T}:T\oplus T^{\ast}\to T$ are
bijective maps sending the metrics $<.\,,.>|E'$ and $<.\,,.>|E''$ to $g$ and $-g$. A generalized complex structure
${\mathcal J}$ on $T$ is called compatible with $E$ if ${\mathcal J}E=E$; in this case ${\mathcal J}E''=E''$. Define a
generalized complex structure ${\mathcal J}^2$ on $T$ by ${\mathcal J}^2={\mathcal J}$ on $E'$, ${\mathcal
J}^2=-{\mathcal J}$ on $E''$, and set ${\mathcal J}^1={\mathcal J}$. Then $({\mathcal J}^1,{\mathcal J}^2)$ is a pair of
commuting generalized complex structures for which the metric
$<-\mathcal{J}^1\circ\mathcal{J}^2(v),w>=<\mathcal{J}^2(v),\mathcal{J}^1(w)>$ on $T\oplus T^{\ast}$ is positive definite.
Recall that such a pair is called linear generalized K\"ahler structure \cite{Gu, Gu14}. Conversely, for every linear
generalized K\"ahler structure $({\mathcal J}^1,{\mathcal J}^2)$, the $+1$-eigenspace of the involution $-{\mathcal
J}^1{\mathcal J}^2$ is a generalized metric compatible with ${\mathcal J}^1$. Note also that if $g$ is a positive
definite metric on $T$, then a generalized complex structure on $T$ is compatible with the generalized metric
$E=\{X+\imath_{X}g: X\in T\}$ if and only if it is compatible with the metric on $T\oplus T^{\ast}$ induced by $g$.

A generalized metric on a manifold $M$ is a subbundle $E$ of $TM\oplus T^{\ast}M$ such that $rank\,E=dim\,M$ and the
metric $<.\,,.>$ is positive definite on $E$. Given a generalized metric $E$, denote by ${\mathcal G}(E)$ the bundle over
$M$ whose fibre at every point $p\in M$ consists of all generalized complex structures on the tangent space $T_pM$
compatible with the generalized metric $E_p$, the fibre of $E$ at $p$. Equivalently, the fibre of ${\mathcal G}(E)$ is
the set of  linear generalized K\"ahler structures on $T_pM$ yielding the given generalized metric $E_p$. We call
${\mathcal G}(E)$ the generalized twistor space of the generalized Riemannian manifold $(M,E)$. Let ${\mathcal Z}(E')$ be
the bundle over $M$ whose fibre at $p\in M$ consists of complex structures on the vector space $E'_p$ compatible with the
metric $g'=<.\,,.>|E'$. Similarly, let ${\mathcal Z}(E'')$ be the bundle of complex structures on the spaces $E''_p$
compatible with the positive definite metric \\  $g''=-<.\,,.>|E''$. Then the bundle ${\mathcal G}(E)$ is isomorphic to
the product bundle ${\mathcal Z}(E')\times {\mathcal Z}(E'')$. Given connections $D'$ and $D''$ on the bundles $E'$ and
$E''$ one can define a generalized almost complex structure ${\mathcal J}_1$ on ${\mathcal G}(E)$ following the general
scheme of the twistor construction. This structure is an analog of the Atiyah-Hitchin-Singer  almost complex structure on
the usual twistor space \cite{AtHiSi}. One can also define three generalized almost complex structures ${\mathcal J}_i$,
$i=2,3,4$, on ${\mathcal G}(E)$ which are analogs of the Eells-Salamon  almost complex structure \cite{EeSa}. As one can
expect, the structures ${\mathcal J}_i$ are never integrable. As far as ${\mathcal J}_1$ is concerned, we discuss the
integrability conditions for ${\mathcal J}_1$ in the case when the connections $D'$ and $D''$ are determined by the
generalized metric $E$ as follows. Using the Courant bracket one can define a metric connection $D'$ on the bundle $E'=E$
\cite{Hit10}. Transferring this connection by means of the isomorphism $pr_{TM}|E:E\to TM$ we obtain a connection
$\nabla$ on $TM$ compatible with the metric $g$ whose torsion $3$-form is $d\Theta$ , $g$ and $\Theta$ being the metric
and the $2$-form determined by $E$ [ibid.]. The connection on $TM\oplus T^{\ast}M$ induced by $\nabla$ may not preserve
the bundle $E''$, so we transfer $\nabla$ to a connection $D''$ on $E''$ by means of the isomorphism
$(pr_{TM}|E'')^{-1}:TM\to E''$. The manifold ${\mathcal G}(E)$ has four connected components and we find the
integrability conditions for the restriction of ${\mathcal J}_1$ to each of these components when $dim\,M=4k$. One of the
integrability conditions is $d\Theta=0$ and the others impose restrictions on the curvature of the Riemannian manifold
$(M,g)$. In the case of an oriented four-dimensional manifold $M$ these curvature restrictions coincide with those found
in \cite{Des} when $\Theta=0$. The reason is that if $d\Theta=0$, $\nabla$ is the Levi-Civita connection of $(M,g)$ used
therein to define the twistor space. Another explanation of this fact is that if $\Theta$ is closed,  the generalized
almost complex structures corresponding to the generalized metrics $E=\{X+\imath_{X}g+\imath_{X}\Theta: X\in TM\}$ and
$\widehat{E}=\{X+\imath_{X}g: X\in TM\}$ are equivalent (see Sec. 7).

A specific property of generalized twistor spaces that the usual twistor spaces do not possess is that the generalized
twistor spaces admit naturally defined (intrinsic) isomorphisms. One of these reflects the so-called $B$-transforms (the
latter being an important feature of the generalized geometry), the others come from the decomposition $TM\oplus
T^{\ast}M=E'\oplus E''$. In particular, if $E$ and $\widehat{E}$ are generalized metrics on a manifold determined by the
same metric $g$ and $2$-forms $\Theta$, $\widehat\Theta$ such that the $2$-form $\Theta-\widehat\Theta$ is closed, the
natural generalized almost complex structures on the generalized twistor spaces ${\mathcal G}(E)$ and ${\mathcal
G}(\widehat{E})$ are equivalent.

This paper is organized as follows. In  Section 2, we collect
several known facts for generalized geometry used in the paper. The
generalized almost complex structures ${\mathcal J}_{\varepsilon}$,
$\varepsilon=1,...,4$, on ${\mathcal G}(E)$ mentioned above are
defined in the third section. The fourth one contains technical
lemmas needed for computing the Nijenhuis tensors of the structures
${\mathcal J}_{\varepsilon}$. Coordinate-free formulas for the
Nijensuis tensors are given in Section 5. These formulas are used in
Section 6 to obtain integrability conditions for ${\mathcal
J}_{\varepsilon}$. Section 7 is devoted to natural isomorphisms of
generalized twistor spaces.

\smallskip

\noindent {\bf Acknowledgment}. I would like to thank the referee whose remarks and comments helped to improve the final
version of the paper.

\section{Preliminaries}

\subsection{Generalized complex structures on vector spaces}

Let $T$ be a $n$-dimensional real vector space. Suppose we are given
a metric $g$ and a complex structure $J$ on $T$.  Let
$J^{\ast}:T^{\ast}\to T^{\ast}$ be the dual map of $J$. Then the
complex structure $J$ is compatible with $g$, i.e. $g$-orthogonal,
if and only if $J=-J^{\ast}$ under the identification $T\cong
T^{\ast}$ determined by the metric $g$. Replacing $T$ by the vector
space $T\oplus T^{\ast}$,  note that we have a canonical isomorphism
$T\oplus T^{\ast}\cong (T\oplus T^{\ast})^{\ast}$.

\noindent {\bf Definition}. A generalized complex structure on $T$
is a complex structure ${\mathcal J}$ on the space $T\oplus
T^{\ast}$ such that ${\mathcal J}=-{\mathcal J}^{\ast}$ under the
identification $T\oplus T^{\ast}\cong (T\oplus T^{\ast})^{\ast}$.

The latter isomorphism is determined by the metric
$<X+\alpha,Y+\beta>=\alpha(Y)+\beta(X)$, $X,Y\in T$,
$\alpha,\beta\in T^{\ast}$, of signature $(n,n)$. Thus the condition
${\mathcal J}=-{\mathcal J}^{\ast}$ is equivalent to the requirement
that ${\mathcal J}$ is compatible with this metric.  It turns out that it is convenient to consider
one half of that metric, so we set
$$
<X+\alpha,Y+\beta>=\frac{1}{2}[\alpha(Y)+\beta(X)],\quad X,Y\in
T,\quad \alpha,\beta\in T^{\ast}.
$$

We note also
that if a real vector space admits a generalized complex structure
it is of even dimension \cite{Gu}.
\smallskip

\noindent {\it Notation}. The map $T\to T^{\ast}$ determined by a
bilinear form $\varphi$ on $T$ will be denoted again by $\varphi$;
thus $\varphi(X)(Y)=\varphi(X,Y)$.

\smallskip

Here are some standard examples of generalized complex structures \cite{Gu,Gu11}.

\noindent {\bf Examples}. {\bf 1}. Every complex structure $J$ on $T$
determines a generalized complex structure $\mathcal{J}$ defined by
$$
{\mathcal J}X = JX,\quad{\mathcal J}\alpha=
-J^{\ast}\alpha\quad\textrm{for}\quad X\in T,\, \alpha \in T^{\ast}.
$$
\noindent {\bf 2}. If $\omega$ is a symplectic form on  $T$ (a non-degenerate
skew-symmetric $2$-form), the map $\omega:T\to T^{\ast}$ is
an isomorphism and we set
$$
\mathcal{J}X=-\omega(X),\quad
\mathcal{J}\alpha=\omega^{-1}(\alpha).
$$
Then $\mathcal{J}$ is a generalized complex structure on $T$.

\smallskip

\noindent {\bf 3}. Let $J$ be a complex structure on $T$. Let $T^{\mathbb
C}=T^{1,0}\oplus T^{0,1}$ be the decomposition of the
complexification of $T$ into the direct sum of $(1,0)$ and
$(0,1)$-vectors with respect to $J$. Take a $2$-vector $\pi \in
\Lambda^{2}T^{\mathbb C}$. Then, for  $\xi \in (T^{1,0})^{\ast}$,
there is a unique vector $\pi^{\sharp}(\xi)\in T^{{\mathbb C}}$ such
that
$$
\eta(\pi^{\sharp}(\xi)) = (\xi \wedge \eta)(\pi)\quad \textrm{for
every}\quad \eta \in (T^{1,0})^{\ast}.
$$
In fact $\pi^{\sharp}(\xi)\in T^{1,0}$ and depends only on the
$\Lambda^2T^{1,0}$-component of $\pi$. Then we can define a
generalized complex structure ${\mathcal J}$ on $T$ setting
$$
{\mathcal J}X=JX+2(Im\,\pi^{\sharp})(\alpha),\quad {\mathcal
J}\alpha=-J^{\ast}\alpha,
$$
where ($Im\,\pi^{\sharp})(\alpha)$ is the vector in $T$ determined
by the identity
$\beta((Im\,\pi^{\sharp})(\alpha))=(\alpha\wedge\beta)(Im\,\pi)$ for
every $\beta\in T^{\ast}$.

\smallskip

\noindent{\bf 4}. The direct sum of generalized complex structures
is also a generalized complex structure.

\smallskip

\noindent{\bf 5}. Any $2$-form $B\in \Lambda^{2}T^{\ast}$ acts on $T\oplus T^{\ast}$ via the inclusion
$\Lambda^{2}T^{\ast}\subset \Lambda^{2}(T\oplus T^{\ast})\cong so(T\oplus T^{\ast})$; in fact this is the action
$X+\alpha\to B(X)$,~ $X\in T$, $\alpha\in T^{\ast}$. Denote the latter map again by $B$. Then the invertible map $e^{B}$
is given by $X+\alpha\to X+\alpha+B(X)$ and is an orthogonal transformation of $T\oplus T^{\ast}$ called {\it a
$B$-transform}. Thus, given a generalized complex structure ${\mathcal J}$ on $T$, the map $e^{B}{\mathcal J}e^{-B}$ is
also a generalized complex structure on $T$, called {\it the $B$-transform of ${\mathcal J}$}.

\smallskip

We refer to \cite{Gu,Gu11} for more linear algebra of generalized
complex structures on vector spaces.

\subsection{Generalized metrics on vector spaces}

Let $T$ be a $n$-dimensional real vector space.  Every metric $g$ on
$T$ is completely determined by its graph $E=\{X+g(X):~X\in
T\}\subset T\oplus T^{\ast}$. The restriction to $E$ of the metric
$<.\,,.>$ on $T\oplus T^{\ast}$ is
$$
<X+g(X),Y+g(Y)>=g(X,Y).
$$
In particular, $<.\,,.>|E$ is positive definite if $g$ is so. This
motivates the following definition \cite{Gu,Witt}.

\noindent{\bf Definition}\label{G-W} A generalized metric on $T$ is
a subspace $E$ of $T\oplus T^{\ast}$ such that

\noindent $(1)$~$dim\,E=dim\,T$

\noindent $(2)$~The restriction of the metric $<.\,,.>$  to $E$ is
positive definite.

\smallskip

Set
$$
E'=E, \quad E''=E^{\perp}=\{w\in T\oplus
T^{\ast}:~<w,v>=0~~\textrm{for~ every}~~v\in E\}.
$$
Then $T\oplus T^{\ast}=E'\oplus E''$ since the bilinear form
$<.\,,.>$ is non-degenerate. Moreover the metric $<.\,,.>$ is
negative definite on $E''$.

It is easy to see that to determine a generalized metric on $T$ is
equivalent to defining  an orthogonal, self-adjoint with respect to
the metric $<.\,,.>$,  linear operator $\mathscr{G}:T\oplus
T^{\ast}\to T\oplus T^{\ast}$ such that $<\mathscr{G}w,w>$ is
positive for $w\in T\oplus T^{\ast}$, $w\neq 0$. Such an operator
$\mathscr{G}$ is an involution different from $\pm$ the identity and
the generalized metric corresponding to it is the $+1$-eigenspace of
$\mathscr{G}$.

If  $E$ is a generalized metric, we have $T^{\ast}\cap E=\{0\}$
since the restriction of the metric $<.\,,.>$ to $T^{\ast}$
vanishes, while its restriction to $E$ is positive definite. Thus
$T\oplus T^{\ast}=E\oplus T^{\ast}$ since
$dim\,E=dim\,T^{\ast}=n$. Then $E$ is the graph of a map
$\alpha:T\to T^{\ast}$, $E=\{X+\alpha(X):~X\in T\}$. Let   $g$ and
$\Theta$ be the bilinear forms on $T$ determined by the symmetric
and skew-symmetric parts of $\alpha$. Under this notation
\begin{equation}\label{gm}
E'=E=\{X+g(X)+\Theta(X):~X\in T\},\quad E''=\{X-g(X)+\Theta(X):~X\in
T\}.
\end{equation}
The restriction of the metric $<.\,, .>$ to $E$ is
\begin{equation}\label{pd}
<X+g(X)+\Theta(X),Y+g(Y)+\Theta(Y)>=g(X,Y),\quad X,Y\in T.
\end{equation}
Hence the bilinear form $g$ on $T$ is positive definite. Thus every generalized metric $E$ is uniquely determined by a
positive definite metric $g$ and a skew-symmetric $2$-form $\Theta$ on $T$ such that $E$ has the representation
(\ref{gm}). Let $pr_{T}: T\oplus T^{\ast}\to T$ be the natural projection. The restriction of this projection to $E$ is
an isomorphism since $E\cap T^{\ast}=\{0\}$. Identity (\ref{pd}) tells us that the isomorphism $pr_{T}|E:E\to T$ is an
isometry when $E$ is equipped with the metric $<.\,.>|E$ and $T$ with the metric $g$. Similarly, the map $pr_{T}|E''$ is
an isometry of the metrics $<.\,,.>|E''$ and $-g$.

\subsection{Generalized Hermitian structures on vector spaces}
Let $E=\{X+g(X):~X\in T\}$ be the generalized metric determined by a positive definite metric $g$ on $T$ and let
${\mathcal J}$ be the generalized complex structure determined by a complex structure $J$ on $T$, $\mathcal{J}X=JX$,
$\mathcal{J}\alpha=-J^{\ast}\alpha$, $X\in T$, $\alpha\in T^{\ast}$. Then $J$ is compatible with $g$, i.e.
$g$-orthogonal, if and only if ${\mathcal J}E\subset E$ (and so ${\mathcal J}E=E$). This leads to the following
definition, see \cite{Gu}.

\noindent {\bf Definition}. A generalized complex structure
$\mathcal{J}$ on $T$ is said to be compatible with a generalized
metric $E$ if the operator $\mathcal{J}$ preserves the space $E$.

As usual, if ${\mathcal J}$ is compatible with $E$, we shall also
say that the generalized metric $E$ is compatible with ${\mathcal
J}$. A pair $(E,{\mathcal J})$ of a  generalized metric and  a
compatible  generalized complex structure is said to be a {\it
generalized Hermitian structure}.

Suppose that a generalized metric $E$ is determined by  an
orthogonal, self-adjoint  linear operator $\mathscr{G}:T\oplus
T^{\ast}\to T\oplus T^{\ast}$ with the property that
$<\mathscr{G}w,w>$ is positive for  $w\neq 0$. Then a generalized
complex structure ${\mathcal J}$ is compatible with $E$ if and only
if the linear operators $\mathcal{J}$ and $\mathscr{G}$ commute. In
this case $\mathcal{J}^2=\mathscr{G}\circ \mathcal{J}$ is a
compatible generalized complex structure on $T$ commuting with the
generalized complex structure $\mathcal{J}^1= \mathcal{J}$.
Moreover, the metric
$$
<-\mathcal{J}^1\circ\mathcal{J}^2(v),w>=<\mathcal{J}^2(v), \mathcal{J}^1(w)>
$$
on $T\oplus T^{\ast}$ is positive definite.  Recall that a pair of
$(\mathcal{J}^1,\mathcal{J}^2)$ of commuting generalized complex
structures such that the metric above is positive definite is called
a {\it linear generalized K\"ahler structure} \cite{Gu, Gu14}. Given
such a structure, the operator
$\mathscr{G}=-\mathcal{J}^1\circ\mathcal{J}^2$ determines a
generalized metric compatible with ${\mathcal J}^1$ and ${\mathcal
J}^2$. Thus the notion of a generalized Hermitian structure on a
vector space is equivalent to the concept of a linear generalized
K\"ahler structure. To fix a generalized metric $E$ means to
consider a linear generalized K\"ahler structure
$(\mathcal{J}^1,\mathcal{J}^2)$ such that $E$ is the $+1$-eigenspace
of the involution $\mathscr{G}=-\mathcal{J}^1\circ\mathcal{J}^2$.

\smallskip

\noindent {\bf Example} {\bf 6}. Let $J$ be a complex structure on
$T$ compatible with a metric $g$ and let $\omega(X,Y)=g(X,JY)$. If
${\mathcal J}^1$ and ${\mathcal J}^2$ are the generalized complex
structures determined by $J$ and $\omega$, respectively, then
$({\mathcal J}^1,{\mathcal J}^2)$ is a linear generalized K\"ahler
structure. The generalized Hermitian structure  defined by
$({\mathcal J}^1,{\mathcal J}^2)$  is $(E,{\mathcal J}^1)$, where
$$
E=\{X+\alpha\in T\oplus T^{\ast}:~{\mathcal J}^1(X+\alpha)={\mathcal J}^2(X+\alpha)\}=\{X+g(X):X\in T\}.
$$
This generalized metric is determined by the operator
$\mathscr{G}=-{\mathcal J}^1\circ{\mathcal J}^2$; it is given by
$\mathscr{G}(X+g(Y))=Y+g(X)$, $X,Y\in T$.

\smallskip

Let $(E,{\mathcal J})$ be a generalized Hermitian structure with
$E=\{X+g(X)+\Theta(X):~X\in T\}$. Then ${\mathcal J}E'=E'$,
${\mathcal J}E''=E''$, where, as above, $E'=E$, $E''=E^{\perp}$, and
we can define two complex structures on $T$ setting
\begin{equation}\label{J1,2}
J_1=(pr_{T}|E')\circ\mathcal{J}\circ (pr_{T}|E')^{-1},\quad
J_2=(pr_{T}|E'')\circ\mathcal{J}\circ (pr_{T}|E'')^{-1}.
\end{equation}
These structures are compatible with the metric $g$. Thus we can
assign a positive definite metric $g$, a skew-symmetric form
$\Theta$ and two $g$-compatible complex structures $J_1$ , $J_2$ on
$T$ to any generalized Hermitian structure $(E,{\mathcal J})$.  The
generalized complex structure ${\mathcal J}$ can be reconstructed
from the the data $g,\Theta,J_1,J_2$ by means of an explicit formula
\cite{Gu}.

\begin{prop}\label{J-J1,2}
Let $g$ be a positive definite metric, $\Theta$ - a skew-symmetric
$2$-form on $T$, and $J_1$, $J_2$ - two complex structures
compatible with the metric $g$. Let $\omega_{k}(X,Y)=g(X,J_{k}Y)$ be
the fundamental $2$-forms of the Hermitian structure $(g,J_{k})$,
$k=1,2$. Then the block-matrix representation of the generalized
complex structure $\mathcal{J}$ determined by  the data
$(g,\Theta,J_1,J_2)$ is of the form
$$
\mathcal{J}=\frac{1}{2}\left(
  \begin{array}{cc}
    I & 0 \\
    \Theta & I \\
  \end{array}
  \right)
  \left(
  \begin{array}{cc}
    J_1+J_2 & \omega_{1}^{-1}-\omega_{2}^{-1} \\
    -(\omega_{1}-\omega_{2})& -(J_1^{\ast}+J_2^{\ast}) \\
  \end{array}
  \right)
    \left(
  \begin{array}{cc}
    I & 0 \\
    -\Theta & I \\
  \end{array}
  \right),
  $$
where $I$ is the identity matrix and $\Theta$, $\omega_1$,
$\omega_2$ stand for the maps $T\to T^{\ast}$ determined by the
corresponding $2$-forms.
\end{prop}
This follows from the identities $\omega_k^{-1}\circ g=J_k$, $\omega_k=-g\circ J_k$, $J^{\ast}_k\circ g=-g\circ J_k$,
$k=1,2$, and the following facts, which will be used further on: \\(a) the $E'$ and $E''$-components
 of a vector $X\in T$ are
\begin{equation}\label{prX}
\begin{array}{c}
X_{E'}=\displaystyle{\frac{1}{2}}\{X- (g^{-1}\circ \Theta)(X)+ g(X)-
(\Theta\circ g^{-1}\circ \Theta)(X)\},\\[8pt]
X_{E''}=\displaystyle{\frac{1}{2}}\{X+ (g^{-1}\circ \Theta)(X)-
g(X)+ (\Theta\circ g^{-1}\circ \Theta)(X)\};
\end{array}
\end{equation}
the  components  of a $1$-form $\alpha\in T^{\ast}$ are given by
\begin{equation}\label{pralpha}
\begin{array}{c}
\alpha_{E'}=\displaystyle{\frac{1}{2}}\{g^{-1}(\alpha)+\alpha+(\Theta\circ
g^{-1})(\alpha)\},\\[8pt]
\alpha_{E''}=\displaystyle{\frac{1}{2}}\{-g^{-1}(\alpha)+\alpha-(\Theta\circ
g^{-1})(\alpha)\}.
\end{array}
\end{equation}
(b)$
\begin{array}{c}
{\mathcal J}(X+g(X)+\Theta(X))=J_1X+g(J_1X)+\Theta(J_1X),\\[6pt]
{\mathcal J}(X-g(X)+\Theta(X))=J_2X-g(J_2X)+\Theta(J_2X).
\end{array}
$

\smallskip

\smallskip

\noindent {\bf Example} {\bf 7}. Let $J$ be a complex structure on
$T$ compatible with a metric $g$. Then, under the notation in the
proposition above, ${\mathcal J}$ is the generalized complex
structure defined by $J$ exactly when $J_1=J_2=J$ and $\Theta=0$.
The generalized complex structure defined by the $2$-form
$\omega(X,Y)=g(X,JY)$ is determined by the data
$(g,\Theta=0,J=J_1=-J_2)$.

\smallskip

\noindent {\bf Remarks}. 1. The forms $\omega_{k}$ used here
differ by a sign from those used in \cite[Proposition 6.12]{Gu}.

2. Suppose that the generalized complex structure ${\mathcal J}$ is
determined by the data $(g,\Theta,J_1,J_2)$. Let $\mathscr{G}$ be
the endomorphism of $T\oplus T^{\ast}$ corresponding to the
generalized metric defined by means of $(g,\Theta)$. Then the
generalized complex structure ${\cal J}^2=\mathscr{G}\circ {\mathcal
J}$ is determined by the data $(g,\Theta,J_1,-J_2)$.

3. It follows from
(\ref{prX}) and (\ref{pralpha}) that the block-matrix representation
of the endomorphism $\mathscr{G}$ is (\cite{Gu})
$$
\mathscr{G}=\left(
  \begin{array}{cc}
    I & 0 \\
    \Theta & I \\
  \end{array}
  \right)
  \left(
  \begin{array}{cc}
    0 & g^{-1} \\
    g & 0 \\
  \end{array}
  \right)
    \left(
  \begin{array}{cc}
    I & 0 \\
    -\Theta & I \\
  \end{array}
  \right).
  $$

4. According to Proposition~\ref{J-J1,2}, ${\mathcal J}=e^{\Theta}{\mathcal I}e^{-\Theta}$ where
${\mathcal I}$ is the generalized complex structure on $T$ with block-matrix
$$
\mathcal{I}=\frac{1}{2}
  \left(
  \begin{array}{cc}
    J_1+J_2 & \omega_{1}^{-1}-\omega_{2}^{-1} \\
    -(\omega_{1}-\omega_{2})& -(J_1^{\ast}+J_2^{\ast}). \\
  \end{array}
  \right).
 $$
The restriction to $T^{\ast}$ of every $B$-transform of $T\oplus
T^{\ast}$ is the identity map. It follows that ${\mathcal J}$
preserves $T^{\ast}$ exactly when $J_1=J_2$ and ${\mathcal J}$ sends
$T^{\ast}$ into $T$ if and only if $J_1=-J_2$. Thus, if $J_1\neq
J_2$, the generalized complex structure ${\mathcal J}$ is not a
$B$-transform of the generalized complex structure determined by a
complex structure (Example 1), or by a complex structure and a
$2$-vector (Example 3). Also, if $J_{1}\neq - J_{2}$, ${\mathcal J}$
is not a $B$-transform of the generalized complex structure
determined by a symplectic form (Example 2).

\smallskip

 Proposition 1 in \cite{Des} and the fact that to define a
generalized Hermitian structure is equivalent to defining a linear
generalized K\"ahler structure imply the following

\begin{prop}
Let $g$ be a positive definite metric on $T$ and $g^{\ast}$ the
metric on $T^{\ast}$ determined by $g$. A generalized complex
structure $\mathcal{J}$ on $T$ is compatible with the generalized
metric $E=\{X+g(X):~ X\in T\}$ if and only if it is compatible with
the metric $g\oplus g^{\ast}$ on $T\oplus T^{\ast}$.
\end{prop}

This can also be proved by means of (\ref{prX}) and (\ref{pralpha}).

\subsection{Generalized almost complex structures on manifolds. The Courant bracket}

{\it A generalized almost complex structure} on an even-dimensional smooth manifold $M$
is, by definition, an endomorphism ${\mathcal J}$ of the bundle $TM\oplus T^{\ast}M$ with
${\mathcal J}^2=-Id$ which preserves the natural  metric
\begin{equation}\label{metric}
<X+\alpha,Y+\beta>=\frac{1}{2}[\alpha(Y)+\beta(X)],\quad X,Y\in
TM,\quad \alpha,\beta\in T^{\ast}M.
\end{equation}
Such a structure is said to be {\it integrable} or {\it a
generalized complex structure} if its $+i$-egensubbunle of
$(TM\oplus T^{\ast}M)\otimes {\Bbb C}$ is closed under the Courant
bracket \cite{Hit02}. Recall that if $X,Y$ are vector fields on $M$
and $\alpha,\beta$ are $1$-forms, the Courant bracket \cite{Cou} is
defined by the formula
$$
[X+\alpha,Y+\beta]=[X,Y]+{\mathcal L}_{X}\beta-{\mathcal
L}_{Y}\alpha-\frac{1}{2}d(\imath_X\beta-\imath_Y\alpha),
$$
where $[X,Y]$ on the right hand-side is the Lie bracket, ${\mathcal L}$ means the Lie
derivative, and $\imath$ stands for the interior product. Note that the Courant bracket is skew-symmetric
like the Lie bracket but it does not satisfy the Jacobi identity.

\smallskip

\noindent {\bf Examples} \cite{Gu,Gu11}. {\bf 8}. The generalized
complex structure defined by an almost complex structure $J$ on $M$
is integrable if and only if $J$ is integrable.

{\bf 9}. The generalized complex structure determined by a
pre-symplectic form $\omega$ is integrable if and only if $\omega$
is symplectic, i.e. $d\omega=0$.

{\bf 10}. Let $J$ be an almost complex manifold on $M$ and $\pi$ a
(smooth) section of $\Lambda^2T^{1,0}M$. The generalized almost
complex structure $\mathcal{J}$ on $M$ defined by means of $J$ and
$\pi$ is integrable if and only if the almost complex structure $J$
is integrable and the field $\pi$ is holomorphic and Poisson.

\smallskip

As in the case of almost complex structures, the integrability condition
for a generalized almost complex structure ${\mathcal J}$ is equivalent to the vanishing of its
Nijenhuis tensor $N$, the latter being defined by means of the Courant bracket:
$$
N(A,B)=-[A,B]+[{\mathcal J}A,{\mathcal J}B]-{\mathcal J}[{\mathcal
J}A,B]-{\mathcal J}[A,{\mathcal J}B],
$$
where $A$ and $B$ are sections of the bundle $TM\oplus T^{\ast}M$.

Clearly $N(A,B)$ is skew-symmetric. That $N$ is a tensor, i.e. $N(A,fB)=fN(A,B)$ for every smooth function $f$ on $M$,
follows from the following property of the Courant bracket \cite[Proposition 3.17]{Gu}.

\begin{prop}\label{Courant-f}
If $f$ is a smooth function on $M$, then for every sections $A$ and $B$ of $TM\oplus T^{\ast}M$
$$
[A,fB]=f[A,B]+(Xf)B-<A,B>df,
$$
where $X$ is the $TM$-component of $A$.
\end{prop}

 Let ${\mathcal J}$ be a generalized almost complex structure on a
manifold $M$ and let $\Theta$ be a (skew-symmetric) smooth $2$-form
on $M$. Then, according to Example 5, $e^{\Theta}{\mathcal
J}e^{-\Theta}$ is a generalized almost complex structure on $M$. The
exponential map $e^{\Theta}$ is an automorphism of the Courant
bracket (i.e. $[e^{\Theta}A,e^{\Theta}B]=e^{\Theta}[A,B]$) if and
only if the form $\Theta$ is closed. This key property of the
Courant bracket follows from the following formula given in the
proof of \cite[Proposition 3.23]{Gu}.

\begin{prop}\label{Courant-B-transf}
If $\Theta$ is a $2$-form on $M$, then for every sections  $A=X+\alpha$ and $B=Y+\beta$ of $TM\oplus T^{\ast}M$
$$
[e^{\Theta}A,e^{\Theta}B]=e^{\Theta}[A,B]-\imath_{X}\imath_{Y}d\Theta.
$$
\end{prop}
Thus, if the form $\Theta$ is closed,  the structure
$e^{\Theta}{\mathcal J}e^{-\Theta}$ is integrable exactly when the
structure ${\mathcal J}$ is so.

\smallskip

The diffeomorphisms also give symmetries of the Courant bracket \cite{Gu}.

\begin{prop}\label{Courant-diffeo}
If $f:M\to N$ is a diffeomorphism, then the Courant bracket is invariant under the bundle isomorphism
$F=f_{\ast}\oplus (f^{-1})^{\ast}:TM\oplus T^{\ast}M\to TN\oplus T^{\ast}N$:
$$
[F(A),F(B)]=F([A,B]),\quad A,B\in TM\oplus T^{\ast}M.
$$
\end{prop}

Thus, if ${\mathcal J}$ is a generalized almost complex structure on $M$ and $f:M\to N$ is a diffeomorphism,
then $F\circ {\mathcal J}\circ F^{-1}$ is a generalized almost complex structure, which is integrable if and only
if ${\mathcal J}$ is so.

\smallskip

Another important property of the Courant bracket is the following formula proved in \cite[Proposition 3.18]{Gu}.

\begin{prop}\label{d-metric}
Let $A$, $B$, $C$ be sections of the bundle $TM\oplus T^{\ast}M$ and $X$ the $TM$-component of $A$. Then
$$
X<B,C> = <[A,B] + d<A,B>,C> + <B, [A,C] + d<A,C>>.
$$
\end{prop}

\subsection{Connections induced by a generalized metric}

By definition, a {\it generalized metric} on a manifold $M$ is a
subbundle $E$ of $TM\oplus T^{\ast}M$ such that $rank\,E=dim\,M$ and
the restriction of the metric $<.\,,.>$ to $E$ is positive definite.
Every such a bundle $E$ is uniquely determined by a Riemannian
metric $g$ and a  $2$-form $\Theta$ on $M$. The pair $(M,E)$ will be
called a generalized Riemannian manifold.

Let $E'=E$ be a generalized metric and, as above, denote $E^{\perp}$ by $E''$. For $X\in TM$, set
$$
X''=(pr_{TM}|E'')^{-1}(X)\in E'',
$$
where $pr_{TM}:TM\oplus T^{\ast}M\to TM$ is the natural projection.
It follows from Proposition~\ref{d-metric} that if $B$ and $C$ are sections of the bundle $E$
$$
X<B,C>=<[X'',B]_{E},C>+<B,[X'',C]_E>,
$$
where the subscript $E$ means "the $E$-component with respect to the decomposition $TM\oplus T^{\ast}M=E\oplus E''$". The
latter identity is reminiscent of the condition for a connection on the bundle $E$ to be compatible with the metric
$<.\,,.>$. In fact, we have the following statement \cite {Hit06, Hit10}.

\begin{prop}\label{connection}
If $S$ is a section of $E$, then
$$
\nabla_{X}^{E}S=[X'',S]_E
$$
defines a connection preserving the metric $<.\,,.>$.
\end{prop}

Suppose that $E$ is determined by the Riemannian metric $g$ and the $2$-form $\Theta$, so that $E=\{X+g(X)+\Theta(X):
X\in TM\}$. Transferring the connection $\nabla^E$ from the bundle $E$ to the bundle $TM$ via the isomorphism
$pr_{TM}|E:E\to TM$ we get a connection on $TM$ preserving the metric $g$. Denote this connection by $\nabla$ and let $T$
be its torsion. Then we have \cite {Hit06, Hit10}.:

\begin{prop}\label{torsion}
The torsion $T$ of the connection $\nabla$ is skew-symmetric and is given by
$$
g(T(X,Y),Z)=d\Theta(X,Y,Z),\quad X,Y,Z\in TM.
$$
\end{prop}

Interchanging the roles of $E$ and $E''=\{X-g(X)+\Theta(X): X\in TM\}$ we can get a connection $\nabla''$ on $TM$
preserving the Riemannian metric $g$ and having torsion $T''$ with $g(T''(X,Y),Z)=-d\Theta(X,Y,Z)$.

If we set $\nabla'=\nabla$, then $\frac{1}{2}(\nabla'+\nabla'')$ is a metric connection with vanishing torsion, so
it is the Levi-Civita connection of the Riemannian manifold $(M,g)$.

\subsection{The space of compatible generalized complex structures}

Let $E$ be a generalized metric on the vector space $T$. As above,
set $E'=E$ and $E''=E^{\perp}$, the orthogonal complement being
taken with respect to the metric $<.\,,.>$ on $T\oplus T^{\ast}$.

Suppose that $T$ is of even dimension $n=2m$. Denote by ${\mathcal G}(E)$ the set of generalized complex structures
compatible with $E$. Equivalently, ${\mathcal G}(E)$ is the set of linear generalized K\"ahler structures, which
determine the generalized metric $E$. This (non-empty) set has the structure of an imbedded submanifold of the vector
space $so(n,n)$ of the endomorphisms of $T\oplus T^{\ast}$, which are skew-symmetric with respect to the metric
$<.\,,.>$. The tangent space of ${\mathcal G}(E)$ at a point ${\mathcal J}$ consists of the endomorphisms $V$ of $T\oplus
T^{\ast}$ anti-commuting with ${\mathcal J}$, skew-symmetric w.r.t. $<.\,,.>$ and such that $VE\subset E$. Such an
endomorphism $V$ sends also $E''$ into itself. Note also that the smooth manifold ${\mathcal G}(E)$ admits a natural
complex structure ${\mathscr J}$ given by $V\to {\mathcal J}\circ V$.

For every ${\mathcal J}\in {\cal G}(E)$, the restrictions $J'={\mathcal J}|E'$ and $J''={\mathcal J}|E''$ are complex
structures on the vector spaces $E'$ and $E''$ compatible with the positive definite metrics $g'= <.\,,.>|E'$ and
$g''=-<.\,,.>|E''$, respectively. Denote by $Z(E')$ and $Z(E'')$ the sets of complex structures on $E'$ and $E''$
compatible with the metrics $g'$ and $g''$. Consider these sets with their natural structures of imbedded submanifolds of
the  vector spaces $so(E',g')$ and $so(E'',g'')$, where $so(E',g')$ is, as usual, the space of $g'$-skew-symmetric
endomorphisms of $E'$, and similarly for $so(E'',g'')$. The tangent space of $Z(E')$ at $J'$ is $T_{J'}Z(E')=\{V'\in
so(E',g'):~V'J'+J'V'=0\}$;  similarly for the tangent space $T_{J''}Z(E'')$. Recall that the manifold $Z(E')$ admits a
complex structure ${\mathscr J}'$ defined by $V'\to J'\circ V'$; similarly $V''\to J''\circ V''$ defines a complex
structure ${\mathscr J}''$ on $Z(E'')$. The map ${\mathcal J}\to (J',J'')$ is a  diffeomorphism sending a tangent vector
$V$ at ${\mathcal J}$ to the tangent vector $(V',V'')$ where $V'=V|E'$ and $V''=V|E''$. Thus ${\mathcal G}(E)\cong
Z(E')\times Z(E'')$ admits four complex structure  defined by
$$
\begin{array}{c}
K_1(V',V'')=(J'\circ V',J''\circ V''),\quad K_2(V',V'')=(J'\circ V',-J''\circ V''),\\[8pt]
K_3=-K_2,\quad K_4=-K_1.
\end{array}
$$
Clearly, the map ${\mathcal J}\to (J',J'')$ is biholomorphic with respect to the complex structures ${\mathscr J}$ on
${\mathcal G}(E)$ and  $K_1$ on $Z(E')\times Z(E'')$.

 Let $G'(S_1',S_2')=-\frac{1}{2}Trace_{g'}\,(S_1'\circ S_2')$ be the standard metric on $so(E',g')$ induced by $g'$;
similarly denote by $G''$ the metric on $so(E'',g'')$ induced by $g''$. Then, as is well-known, $(G',{\mathscr J}')$ and
$(G'',{\mathscr J}'')$ are K\"ahler structures on $Z(E')$ and $Z(E'')$, so $(G=G'+G'',K_{\varepsilon})$,
$\varepsilon=1,...,4$, is a K\"ahler structure on ${\mathcal G}(E)$.

Let $g$ and $\Theta$ be the positive definite metric and the skew-symmetric $2$-form on $T$ determined by $E$, so that
$E=\{X+g(X)+\Theta(X):~X\in T\}$. Denote by $Z(T,g)$ the manifold of all complex structures on $T$ compatible with the
metric $g$ considered as an imbedded submanifold of the space $so(g)$ of $g$-skew-symmetric endomorphisms of $T$.  Endow
$Z(T,g)$ with its natural complex structure and compatible metric. For ${\mathcal J}\in {\mathcal G}(E)$, let $J_1$ and
$J_2$ be the $g$-compatible complex structures on $T$ defined by means of ${\mathcal J}$:
$$
J_1=(pr_{T}|E')\circ\mathcal{J}\circ (pr_{T}|E')^{-1},\quad
J_2=(pr_{T}|E'')\circ\mathcal{J}\circ (pr_{T}|E'')^{-1}.
$$
Then the map ${\mathcal J}\to (J_1,J_2)$ is an isometry of
${\mathcal G}(E)$ onto $Z(T,g)\times Z(T,g)$. Moreover it sends a
tangent vector $V$ at ${\mathcal J}\in {\mathcal G}(E)$ to the
tangent vector $(V_1,V_2)$,  where
$$
V_1=(pr_{T}|E')\circ V \circ (pr_{T}|E')^{-1},\quad
V_2=(pr_{T}|E'')\circ V \circ (pr_{T}|E'')^{-1}.
$$
Hence ${\mathcal J}\to (J_1,J_2)$ is a biholomorphic map. The
manifold $Z(T,g)$ has the homogeneous representation $O(2m)/U(m)$
where $2m=dim\,T$ and the group $O(2m)\cong O(g)$ acts by
conjugation. In particular, it has two connected components, each of
them having the homogeneous representation $SO(2m)/U(m)$. Fix an
orientation on the vector space $T$ and denote by $Z_{\pm}$ the
space of complex structures on $T$ compatible with the metric $g$
and yielding $\pm$ the orientation of $T$. Then $Z_{+}$ and $Z_{-}$
are the connected components of $Z(T,g)$. Thus ${\mathcal G}(E)$ has
four connected components  biholomorphically isometric to
$Z_{+}\times Z_{+}$, $Z_{+}\times Z_{-}$, $Z_{-}\times Z_{+}$,
$Z_{-}\times Z_{-}$. If $dim\,T=4k$, the open subsets ${\mathcal
G}_{+}$ and ${\mathcal G}_{-}$ of  ${\mathcal G}(E)$ biholomorphic
to $(Z_{+}\times Z_{+})\cup (Z_{-}\times Z_{-})$ and $(Z_{+}\times
Z_{-})\cup (Z_{-}\times Z_{+})$ can be described in terms of the
generalized complex structures as follows. Recall first that the
vector space $T\oplus T^{\ast}$ has a canonical orientation; if
$\{a_i\}$ is an arbitrary basis of $T$ and $\{\alpha_i\}$ is its
dual basis, $i=1,...,n$, the orientation of the space $T\oplus
T^{\ast}$ defined by the basis $\{a_i,\alpha_i\}$ does not depend on
the choice of the basis $\{a_i\}$. Let $\mathscr{G}$ be the
endomorphism of $T\oplus T^{\ast}$ determined by the generalized
metric $E$. Then, by \cite[Remark 6.14 and Proposition 4.7]{Gu},
${\mathcal J}\in{\mathcal G}_{\pm}(E)$ if and only if  the complex
structures ${\mathcal J}_1={\mathcal J}$ and ${\cal
J}_2=\mathscr{G}\circ {\mathcal J}_1$ both induce $\pm$ the
canonical orientation of $T\oplus T^{\ast}$. We also note that if
$dim\,T=4k+2$, then, by \cite[Propostion 6.8]{Gu}, one of  the
complex structures ${\mathcal J}_1={\mathcal J}\in {\mathcal G}(E)$
and ${\mathcal J}_2=\mathscr{G}\circ {\mathcal J}_1$ induces the
canonical orientation
of $T\oplus T^{\ast}$, while the other one the opposite orientation. 

\section{Generalized twistor spaces}

  Let $M$ be a smooth manifold of dimension $n=2m$ equipped with a generalized metric $E$ determined
by a Riemannian metric $g$ and a $2$-form $\Theta$ on $M$. Denote
by ${\mathcal G}={\mathcal G}(E)\to M$  the bundle over $M$ whose fibre at a point
$p\in M$ consists of all generalized complex structures on $T_pM$
compatible with the generalized metric $E_p$, the fibre of $E$ at $p$.
We call ${\mathcal G}$  generalized twistor space of the generalized Riemannian manifold $(M,E)$.

Set $E'=E$ and $E''=E^{\perp}$, the orthogonal complement of $E$ in
$TM\oplus T^{\ast}M$ with respect to the metric $<.\,,.>$. Denote by
${\mathcal Z}(E')$ the bundle over $M$ whose fibre at a point $p\in
M$ is constituted of all complex structures on the vector space
$E'_p$ compatible with the positive definite metric $g'=<.\,,.>|E'$.
Define a bundle ${\mathcal Z}(E'')$ in a similar way,  $E''$  being
endowed with  the metric $g''=-<.\,,.>|E''$. Then ${\mathcal G}$ is
identified with the product bundle ${\mathcal Z}(E')\times {\mathcal
Z}(E'')$ by the map ${\mathcal G}_p\ni J\to (J|E'_p,J|E''_p)$.

Suppose we are given metric connections $D'$ and $D''$ on $E'$ and
$E''$, respectively,  and let $D=D'\oplus D''$ be the connection on
$E'\oplus E''=TM\oplus T^{\ast}M$ determined by $D'$ and $D''$.

 The bundle ${\mathcal Z}(E')$ is a subbundle of the vector bundle
$A(E')$ of $g'$-skew-symmetric endomorphisms of $E'$, and similarly
for ${\mathcal Z}(E'')$. Henceforth  we shall consider the bundle
${\mathcal G}\cong {\mathcal Z}(E')\times {\mathcal Z}(E'')$ as a
subbundle of the vector bundle $\pi:A(E')\oplus A(E'')\to M$. The
connection on $A(E')\oplus A(E'')$ induced by the connection
$D=D'\oplus D''$ on $E'\oplus E''$ will again be denoted by $D$. It
is easy to see that the horizontal space of $A(E')\oplus A(E'')$
with respect to $D$ at every point of ${\mathcal G}$ is tangent to
${\mathcal G}$ (cf. the next section). Thus the connection $D$ gives
rise to a splitting ${\cal V}\oplus {\cal H}$ of the tangent bundle
of the bundle ${\cal G}$ into vertical and horizontal parts. Then,
following the standard twistor construction,  we can define four
generalized almost complex structures ${\mathcal J}_{\varepsilon}$
on the manifold ${\mathcal G}$; when we need to indicate explicitly
the bundle $E$ we hall write ${\mathcal J}^{E}_{\varepsilon}$.

 The vertical space ${\mathcal V}_J$ of ${\cal G}$
at a point $J\in {\cal G}$ is the tangent space at $J$ of the fibre
through this point. This fibre is the manifold ${\mathcal
G}(E_{\pi(J)})$, which admits four complex structures
$K_{\varepsilon}$ defined in the preceding section.  We define
${\mathcal J}_{\varepsilon}|({\mathcal V}_J\oplus {\mathcal
V}_J^{\ast})$ to be the generalized complex structure determined by
the complex structure $K_{\varepsilon}$. Thus
$$
{\mathcal
J}_{\varepsilon}=K_{\varepsilon}~ \rm{on}~ {\cal V}_J,\quad   {\mathcal
J}_{\varepsilon}= -K_{\varepsilon}^{\ast} ~\rm{on}~ {\mathcal V}_J^{\ast},\quad
\varepsilon=1,2,3,4.
$$

 The horizontal space ${\cal H}_J$ is isomorphic  to the tangent space $T_{\pi(J)}M$ via the
differential $\pi_{\ast J}$. If $\pi_{{\mathcal H}}$ is the
restriction of $\pi_{\ast}$ to ${\cal H}$, the image of every $A\in
T_pM\oplus T_p^{\ast}M$ under the map $\pi_{{\mathcal H}}^{-1}\oplus
\pi_{{\mathcal H}}^{\ast}$ will be denoted by $A^h$. Thus, for $J\in
{\mathcal G}$, $Z\in T_{\pi(J)}M$ and $\omega\in
T_{\pi(J)}^{\ast}M$, we have $\omega^h_J(Z^h_J)=\omega_{\pi(J)}(Z)$.
The elements of ${\mathcal H}_J^{\ast}$, resp. ${\mathcal
V}_J^{\ast}$, will be considered as $1$-forms on $T_J{\cal G}$
vanishing on ${\mathcal V}_J$, resp. ${\cal H}_J$.

 Now we define a generalized
complex structure ${\mathcal J}$ on the vector space ${\mathcal
H}_J\oplus {\mathcal H}_J^{\ast}$ as the lift of the endomorphism
$J$  of $T_{\pi(J)}M\oplus T_{\pi(J)}^{\ast}M$ by the isomorphism
$\pi_{{\mathcal H}}\oplus (\pi_{{\mathcal H}}^{-1})^{\ast}:{\mathcal
H}_J\oplus {\mathcal H}_J^{\ast}\to T_{\pi(J)}M\oplus
T_{\pi(J)}^{\ast}M$:
$$
{\mathcal J}A^h_J=(JA)^h_J,\quad A\in T_{\pi(J)}M\oplus T^{\ast}_{\pi(J)}M.
$$

Finally,  we set ${\mathcal J}_{\varepsilon}={\mathcal J}$ on
${\mathcal H}\oplus {\mathcal H}^{\ast}$.

\smallskip

\noindent
{\bf Remark} 5. According to Remark 4, if $n\geq 2$, the generalized almost complex structures
${\mathcal J}_{\varepsilon}$ are not $B$-transforms of generalized complex structures induced by complex or
pre-symplectic structures.

\section{Technical lemmas}

To compute the Nijenhuis tensor of the generalized almost complex
structures ${\mathcal J}_{\varepsilon}$, $\varepsilon=1,2,3,4$, on
the twistor space ${\mathcal G}$ we need some preliminary lemmas.

\smallskip

Let $(\mathscr{U},x_1,...,x_{2m})$ be a local coordinate system of
$M$ and $\{Q_1',...,Q_{2m}'\}$, $\{Q_1'',...,Q_{2m}''\}$ orthonormal
frames of $E'$ and $E''$ on $\mathscr{U}$, respectively. Define
sections $S_{ij}'$, $S_{ij}''$, $1\leq i,j\leq {2m}$, of $A(E')$ and
$A(E'')$ by the formulas
\begin{equation}\label{eq Sij}
S_{ij}'Q_k'=\delta_{ik}Q_j' - \delta_{kj}Q_i',\quad S_{ij}''Q_k''=\delta_{ik}Q_j'' - \delta_{kj}Q_i''.
\end{equation}
Then $S_{ij}'$ and  $S_{ij}''$ with $i<j$ form orthonormal frames of
$A(E')$ and $A(E'')$ with respect to the metrics $G'$ and $G''$
defined by
$$G'(a',b')=\displaystyle{-\frac{1}{2}}Trace_{g'}\,(a'\circ b')$$
for $a',b'\in A(E')$, and similarly for $G''$.

For $a=(a',a'')\in A(E')\oplus A(E'')$, set
\begin{equation}\label{coord}
\tilde x_{i}(a)=x_{i}\circ\pi(a),\quad y_{kl}'(a)=G'(a',S_{kl}'\circ\pi(a)),\quad
y_{kl}''(a)=G''(a'',S_{kl}''\circ\pi(a)).
\end{equation}
Then $(\tilde x_{i},y_{jk}', y_{jk}'')$, $1\leq i\leq 2m$, $1\leq j
< k\leq 2m$,  is a local coordinate system on the total space of the
bundle $A(E')\oplus A(E'')$.

Let
\begin{equation}\label{V}
V=\sum_{j<k}[v_{jk}'\frac{\partial}{\partial
y_{jk}'}(J)+v_{jk}''\frac{\partial}{\partial y_{jk}''}(J)]
\end{equation}
be a vertical vector of ${\cal G}$ at a point $J$.  It is convenient
to set $v_{ij}'=-v_{ji}'$, $v_{ij}''=-v_{ji}''$  and
$y_{ij}'=-y_{ji}'$, $y_{ij}''=-y_{ji}''$ for $i\geq j$, $1\leq
i,j\leq {2m}$. Then the endomorphism $V$ of $T_{p}M\oplus
T_{p}^{\ast}M$, $p=\pi(J)$, is determined by
$$
VQ_i'=\sum_{j=1}^{2m} v_{ij}'Q_j',\quad VQ_i''=\sum_{j=1}^{2m}
v_{ij}''Q_j''.
$$
Moreover
\begin{equation}\label{cal J/ver}
{\mathcal J}_{\varepsilon}V=(-1)^{\varepsilon+1}\sum_{j<k}
\sum_{s=1}^{2m}[\pm v_{js}'y_{sk}'\frac{\partial}{\partial y_{jk}'}+
v_{js}''y_{sk}''\frac{\partial}{\partial y_{jk}''}],
\end{equation}
where the plus sign corresponds to $\varepsilon=1,4$ and the minus sign to $\varepsilon=2,3$.

\smallskip

 Note also that, for every $A\in T_{p}M\oplus T_{p}^{\ast}M$, we have

\begin{equation}\label{cal J/hor}
\begin{array}{c}
 A^h=\sum\limits_{i=1}^{2m}[(<A,Q_i'>\circ\pi)Q_i^{'\,h}-(<A,Q_i''>\circ\pi)Q_i^{''\,h}],\\[8pt]
{\mathcal J}A^h=\sum\limits_{i,j=1}^{2m}[(<A,Q_i'>\circ\pi)y_{ij}'Q_j^{'\,h}- (<A,Q_i''>\circ\pi)y_{ij}''Q_j^{''\,h}].
\end{array}
\end{equation}

  For each vector field
$$X=\sum_{i=1}^{2m} X^{i}\frac{\partial}{\partial x_i}$$
on $\mathscr{U}$, the horizontal lift $X^h$ on $\pi^{-1}(\mathscr{U})$ is given by
\begin{equation}\label{Xh}
\begin{array}{c}
X^{h}=\displaystyle{\sum_{l=1}^{2m} (X^{l}\circ\pi)\frac{\partial}{\partial\tilde
x_l}}\\[8pt]
- \displaystyle{\sum_{i<j}\sum_{k<l}
[y_{kl}'(G'(D_{X}S_{kl}',S_{ij}')\circ\pi)\frac{\partial}{\partial
y_{ij}'} +
y_{kl}''(G''(D_{X}S_{kl}'',S_{ij}'')\circ\pi)\frac{\partial}{\partial
y_{ij}''}}].
\end{array}
\end{equation}

     Let $a=(a',a'')\in A(E')\oplus A(E'')$. Denote by $A(E'_{\pi(a)})$ the fiber of $A(E')$ at the point $\pi(a)$
and similarly for $A(E''_{\pi(a)})$. Then (\ref{Xh}) implies that,
under the standard identification of $T_{a}(A(E'_{\pi(a)})\oplus
A(E''_{\pi(a)}))$ with the vector space $A(E'_{\pi(a)})\oplus
A(E''_{\pi(a)})$, we have
\begin{equation}\label{[XhYh]}
[X^{h},Y^{h}]_{a}=[X,Y]^h_a + R(X,Y)a,
\end{equation}
where $R(X,Y)a=R(X,Y)a'+R(X,Y)a''$ is the curvature of the
connection $D$ (for the curvature tensor we adopt the following
definition: $R(X,Y)=D_{[X,Y]}- [D_{X},D_{Y}]$). Note
also that (\ref{V}) and (\ref{Xh}) imply the well-known fact that
\begin{equation}\label{[V,Xh]}
[V,X^h]~\rm{is~ a~ vertical~ vector~ field}.
\end{equation}

\smallskip

\noindent {\it Notation}.  Let $J\in {\mathcal G}$ and $p=\pi(J)$. Take orthonormal bases $\{a_1',...,a_{2m}'\}$,
$\{a_1'',...,a_{2m}''\}$ of $E_p'$, $E_p''$ such that $a_{2l}'=Ja_{2l-1}'$, $a_{2l}''=Ja_{2l-1}''$ for $l=1,...,m$. Let
$\{Q_i'\}$, $\{Q_i''\}$, $i=1,...,2m$, be orthonormal frames of $E'$, $E''$ in the vicinity of the point $p$ such that
$$
Q_i'(p)=a_i',\> Q_i''(p)=a_i'' ~\mbox { and }~ D\, Q_i'|_p=0,\>
D\,Q_i''|_p=0, \quad i=1,...,2m.
$$
Define a section $S=(S',S'')$ of $A(E')\oplus A(E'')$  setting
$$
S'Q_{2l-1}'=Q_{2l}',\quad S''Q_{2l-1}''=Q_{2l}'' ,\quad S'Q_{2l}'=-Q_{2l-1}',\quad S''Q_{2l}''=-Q_{2l-1}'',
$$
 $l=1,...,m$. Then,
$$
S(p)=J, ~ D S|_p=0.
$$
In particular $X^h_J=S_{\ast}X$ for every $X\in T_pM$.

\smallskip

Clearly, the section $S$ takes its values in ${\mathcal G}$, hence
{\it the horizontal space of $A(E')\oplus A(E'')$ with respect to
the connection $D$ at any $J\in {\cal G}$ is tangent to ${\cal G}$}.

 Further on, given a smooth manifold $M$, the natural projections of $TM\oplus T^{\ast}M$
onto $TM$ and $T^{\ast}M$ will be denoted by $\pi_1$ and $\pi_2$,
respectively. The natural projections of ${\mathcal H}\oplus {\mathcal H}^{\ast}$ onto
${\mathcal H}$ and ${\mathcal H}^{\ast}$ will also be denoted by $\pi_1$ and $\pi_2$
when this will not cause confusion. Thus if $\pi_1(A)=X$ for $A\in TM\oplus T^{\ast}M$,
then $\pi_1(A^h)=X^h$ and similarly for $\pi_2(A)$ and $\pi_2(A^h)$.

  We shall use the above notations throughout the next sections.

\smallskip

Note that, although $DS|_p=0$, $D\pi_1(S)$ and $D\pi_2(S)$ may not vanish at the point $p$ since the connection $D$ may
not preserve $TM$ or $T^{\ast}M$.

\smallskip

\begin{lemma}\label {brackets}
If $A$ and $B$ are sections of the bundle $TM\oplus T^{\ast}M$ near $p$, then
\begin{enumerate}
\item[$(i)$]  $[\pi_1(A^h),\pi_1({\mathcal
J}B^h)]_J=[\pi_1(A),\pi_1(SB)]^h_J+R(\pi_1(A),\pi_1(JB))J.$
\vspace{0.2cm}
\item[$(ii)$]
$[\pi_1({\mathcal J}A^h),\pi_1({\mathcal
J}B^h)]_J=[\pi_1(SA),\pi_1(SB)]^h_J+R(\pi_1(JA),\pi_1(JB))J.$
\end{enumerate}
\end{lemma}

\begin{proof}
Set $X=\pi_1(A)$. By (\ref{Xh}), we have  $X_J^h=\sum_{l=1}^{2m} X^l(p)\frac{\partial}{\partial\tilde x_l}(J)$ since $D
S_{kl}'|_p=D S_{kl}''|_p=0$, $k,l=1,...,2m$. Then, using (\ref{cal J/hor}), we get
\begin{equation}\label{XhBh}
\begin{array}{c}
[X^h,\pi_1({\cal J}B^h)]_J=\\[8pt]

\sum\limits_{i,j=1}^{2m}[<B,Q_i'>_py_{ij}'(J)[X^h,\pi_1(Q_j')^h]_J +
X_p(<B,Q_i'>)y_{ij}'(J)(\pi_1(Q_j'))^h_J]\\[8pt]

-\sum\limits_{i,j=1}^{2m}[<B,Q_i''>_py_{ij}''(J)[X^h,\pi_1(Q_j'')^h]_J
+ X_p(<B,Q_i''>)y_{ij}''(J)(\pi_1(Q_j''))^h_J].
\end{array}
\end{equation}
We also have
\begin{equation}\label{sB}
SB=\sum_{i,j=1}^{2m}[<B,Q_i'>(y_{ij}'\circ
S)Q_j'-<B,Q_i''>(y_{ij}''\circ S)Q_j''].
\end{equation}
Therefore
\begin{equation}\label{XB}
\begin{array}{c}
[X,\pi_1(SB)]_p=\\[8pt]

\sum\limits_{i,j=1}^{2m}[<B,Q_i'>_py_{ij}'(J)[X,\pi_1(Q_j')]_p+X_p(<B,Q_i'>)y_{ij}'(J)(\pi_1(Q_j'))_p]\\[8pt]

-\sum\limits_{i,j=1}^{2m}[<B,Q_i''>_py_{ij}''(J)[X,\pi_1(Q_j'')]_p+X_p(<B,Q_i''>)y_{ij}''(J)(\pi_1(Q_j'')]_p].
\end{array}
\end{equation}

Now formula $(i)$ follows from (\ref{XhBh}), (\ref{[XhYh]}) and
(\ref{XB}). A similar computation gives $(ii)$.
\end{proof}

\smallskip

 For any (local) section $a=(a',a'')$ of $A(E')\oplus A(E'')$,
denote by $\widetilde a$ the vertical vector field on ${\cal G}$
defined by
\begin{equation}\label{eq tilde a}
\widetilde a_J=(a'_{\pi(J)}+(J|E')\circ a'_{\pi(J)}\circ (J|E'), a''_{\pi(J)}+(J|E'')\circ
a''_{\pi(J)}\circ (J|E'')) .
\end{equation}
Let us note that for every $J\in {\cal G}$ we can find sections
$a_1,...,a_s$, $s=2(m^2-m)$, of $A(E')\oplus A(E'')$ near the point
$p=\pi(J)$ such that $\widetilde a_1,...,\widetilde a_s$ form a
basis of the vertical vector space at each point in a neighbourhood
of $J$.

\vspace{0.1cm}

\begin{lemma}\label {H-V brackets}
Let $J\in {\cal G}$ and let $a$ be a section of $A(E')\oplus A(E'')$
near the point $p=\pi(J)$. Then, for any section $A$ of the bundle
$TM\oplus T^{\ast}M$ near $p$, we have (for the Lie brackets)
\begin{enumerate}
\item[$(i)$] $[\pi_1(A^h),\widetilde a]_J=(\widetilde{D_{\pi_1(A)}a})_J.$
\vspace{0.2cm}
\item[$(ii)$] $[\pi_1(A^h),{\mathcal J}_{\varepsilon}\widetilde a]_J=
K_{\varepsilon}(\widetilde{D_{\pi_1(A)}a})_J.$
\vspace{0.2cm}
\item[$(iii)$] $[\pi_1({\mathcal J}A^h),\widetilde
a]_J=(\widetilde{D_{\pi_1(JA)}a})_J-(\pi_1(\widetilde a(A)))^h_J.$
\vspace{0.2cm}
\item[$(iv)$] $[\pi_1({\mathcal J}A^h),{\mathcal J}_{\varepsilon}\widetilde
a]_J=K_{\varepsilon}
(\widetilde{D_{\pi_1(JA)}a})_J-(\pi_1((K_{\varepsilon}\widetilde
a)(A)))^h_J.$
\end{enumerate}
\end{lemma}

\begin{proof}
Let $a'(Q_i')=\sum_{j=1}^{2m}a_{ij}'Q_j'$, $a''(Q_i'')=\sum_{j=1}^{2m}a_{ij}''Q_j''$,\, $i=1,...,2m$. Then, in the local
coordinates of $A(E')\oplus A(E'')$ introduced above,
$$
\widetilde a=\sum_{i<j}[\widetilde a_{ij}'\frac{\partial}{\partial
y_{ij}'}+\widetilde a_{ij}''\frac{\partial}{\partial y_{ij}''}],
$$
where
\begin{equation}\label{tilde aij}
\widetilde a_{ij}'=a_{ij}'\circ\pi+\sum_{k,l=1}^{2m}y_{ik}'(a_{kl}'\circ\pi)y_{lj}',\quad
a_{ij}''=a_{ij}''\circ\pi+\sum_{k,l=1}^{2m}y_{ik}''(a_{kl}''\circ\pi)y_{lj}''.
\end{equation}
Let us also note that for every vector field $X$ on $M$ near the
point $p$, we have in view of (\ref{Xh})
$$
\begin{array}{c}
X_J^h=\displaystyle{\sum_{i=1}^{2m}
X^i(p)\frac{\partial}{\partial\tilde
x_i}(J)},\\[8pt]
\displaystyle{[X^h,\frac{\partial}{\partial
y_{ij}'}]_J=[X^h,\frac{\partial}{\partial y_{ij}''}]_J}=0
\end{array}
$$
since $D S_{ij}'|_p=D S_{ij}''|_p=0$. Moreover,
\begin{equation}\label{Xa_ij}
(D_{X_p}a')(Q_i')=\sum_{j=1}^{2m}X_p(a_{ij}')Q_j', \quad
(D_{X_p}a'')(Q_i'')=\sum_{j=1}^{2m}X_p(a_{ij}'')Q_j''
\end{equation}
since $D Q_i'|_p=D Q_i''|_p=0$. Now the lemma follows by simple
computations making use of (\ref{cal J/ver}) and (\ref{cal J/hor}).
\end{proof}

\begin{lemma}\label {Lie deriv}
Let $A$ and $B$ be sections of the bundle $TM\oplus T^{\ast}M$ near
$p$, and let $Z\in T_pM$, $W\in {\cal V}_J$. Then
\begin{enumerate}
\item[$(i)$]
\hspace{0.2cm}$({\mathcal L}_{\pi_1(A^h)}{\pi_2(B^h)})_J=({\mathcal
L}_{\pi_1(A)}{\pi_2(B)})^h_J.$

\vspace{0.3cm}

\item[$(ii)$]
\hspace{0.2cm}$({\mathcal L}_{\pi_1(A^h)}{\pi_2({\mathcal
J}B^h)})_J=({\mathcal L}_{\pi_1(A)}{\pi_2(SB)})^h_J.$

\vspace{0.3cm}

\item[$(iii)$]
$
\begin{array}{lll}
({\mathcal L}_{\pi_1({\mathcal J}A^h)}\pi_2(B^h))_J(Z^h+W)=\\[6pt]
({\mathcal L}_{\pi_1(SA)}\pi_2(B))^h_J(Z^h)+(\pi_2(B))_p(\pi_1(WA)).
\end{array}
$

\vspace{0.3cm}

\item[$(iv)$]
$
\begin{array}{lll}
({\mathcal L}_{\pi_1({\mathcal J}A^h)}\pi_2({\mathcal
J}B^h))_J(Z^h+W)=\\[6pt]
({\mathcal L}_{\pi_1(SA)}\pi_2(SB))^h_J(Z^h)+(\pi_2(JB))_p(\pi_1(WA)).
\end{array}
$
\end{enumerate}
\end{lemma}

\begin{proof}
Formula $(i)$  follows from (\ref{[XhYh]}) and (\ref{[V,Xh]});
$(ii)$ is a consequence of $(i)$, (\ref{cal J/hor}) and (\ref{sB}).
A simple computations involving (\ref{cal J/hor}), (\ref{[XhYh]}),
(\ref{[V,Xh]}) and (\ref{sB}) gives formula $(iii)$; $(iv)$ follows
from $(iii)$, (\ref{cal J/hor}) and (\ref{sB}).
\end{proof}

\vspace{0.1cm}

The proofs of the next lemmas are also easy and will be omitted.

\begin{lemma}\label {Half-diff}
Let $A$ and $B$ are sections of the bundle $TM\oplus T^{\ast}M$ near
$p$. Let $Z\in T_pM$ and $W\in {\cal V}_J$. Then
\begin{enumerate}
\item[$(i)$]
\hspace{0.2cm}$(d\>\imath_{\pi_1(A^h)}\pi_2(B^h))_J=(d\>\imath_{\pi_1(A)}\pi_2(B))^h_J.$

\vspace{0.3cm}

\item[$(ii)$]
$
\begin{array}{lll}
(d\>\imath_{\pi_1(A^h)}\pi_2({\mathcal J} B^h))_J(Z^h+W)= \\[6pt]
(d\>\imath_{\pi_1(A)}\pi_2(SB))^h_J(Z^h)+ (\pi_2(WB))_p(\pi_1(A)).
\end{array}
$

\vspace{0.3cm}

\item[$(iii)$]
$
\begin{array}{lll}
(d\>\imath_{\pi_1({\mathcal J}A^h)}\pi_2(B^h))_J(Z^h+W)= \\[6pt]
(d\>\imath_{\pi_1(SA)}\pi_2(B))^h_J(Z^h)+ (\pi_2(B))_p(\pi_1(WA)).
\end{array}
$

\vspace{0.3cm}

\item[$(iv)$]
$
\begin{array}{lll}
(d\>\imath_{\pi_1({\mathcal J}A^h)}\pi_2({\mathcal J}B^h))_J(Z^h+W)=\\[6pt]
(d\>\imath_{\pi_1(SA)}\pi_2(SB))^h_J(Z^h)+ (\pi_2(WB))_p(\pi_1(JA))+
(\pi_2(JB))_p(\pi_1(WA).
\end{array}
$

\end{enumerate}
\end{lemma}

\begin{lemma}\label {H-V Lie der&diff}
Let $A$ be a section of the bundle $TM\oplus T^{\ast}M$ and $V$ a vertical vector field
on ${\cal G}$. Then
\begin{enumerate}
\item[$(i)$] ${\mathcal L}_{V}\pi_2(A^h)=0$; \hspace{0.2cm}$\imath_{V}\pi_2(A^h)=0$.
\vspace{0.2cm}
\item[$(ii)$] ${\mathcal L}_{V}\pi_2({\mathcal J}A^h)=\pi_2((VA)^h)$;
\hspace{0.2cm} $\imath_{V}\pi_2({\mathcal J}A^h)=0$.
\end{enumerate}
\end{lemma}

\vspace{0.1cm}

\noindent {\it Notation}. Let $J\in {\cal G}$. For any fixed
$\varepsilon=1,...,4$, take a basis
$\{U_{2t-1}^{\varepsilon},U_{2t}^{\varepsilon}={\mathcal
J}_{\varepsilon}U_{2t-1}^{\varepsilon}\}$, $t=1,...,m^2-m$, of the
vertical space ${\cal V}_J$.  Let $a_{2t-1}^{\varepsilon}$ be
sections of $A(E')\oplus A(E'')$ near the point $p=\pi(J)$ such that
$a_{2t-1}^{\varepsilon}(p)=U_{2t-1}^{\varepsilon}$ and $D
a_{2t-1}^{\varepsilon}|_p=0$. Define vertical vector fields
$\widetilde a_{2t-1}^{\varepsilon}$ by (\ref{eq tilde a}). Then
$\{\widetilde a_{2t-1}^{\varepsilon},{\cal
J}_{\varepsilon}\widetilde a_{2t-1}^{\varepsilon}\}$,
$t=1,...,m^2-m$, is a frame of the vertical bundle on ${\mathcal G}$
near the point $J$. Denote by
$\{\beta_{2t-1}^{\varepsilon},\beta_{2t}^{\varepsilon}\}$ the dual
frame of the bundle ${\mathcal V}^{\ast}$. Then
$\beta_{2t}^{\varepsilon}={\cal
J}_{\varepsilon}\beta_{2t-1}^{\varepsilon}$.

\vspace{0.1cm}

  Under these notations, we have the following.


\begin{lemma}\label {H-Vstar Lie der}
Let $A$ be a section of the bundle $TM\oplus T^{\ast}M$ near the
point $p=\pi(J)$. Then for every $Z\in T_pM$, $s,r=1,....,2(m^2-m)$
and $\varepsilon=1,...,4$, we have
\begin{enumerate}
\item[$(i)$]$({\mathcal L}_{\pi_1(A^h)}\beta_s^{\varepsilon})_J(Z^h+U_r^{\varepsilon})=
-\beta_s^{\varepsilon}(R(\pi_1(A),Z)J).$
\vspace{0.2cm}
\item[$(ii)$]$({\mathcal L}_{\pi_1({\cal
J}A^h)}\beta_s^{\varepsilon})_J(Z^h+U_r^{\varepsilon})=-\beta_s^{\varepsilon}(R(\pi_1(JA),Z)J).$
\vspace{0.2cm}
\item[$(iii)$]$({\mathcal L}_{\pi_1(A^h)}{\mathcal J}_{\varepsilon}\beta_s^{\varepsilon})_J(Z^h+U_r^{\varepsilon})=-({\mathcal
J}_{\varepsilon}\beta_s^{\varepsilon})(R(\pi_1(A),Z)J).$
\vspace{0.2cm}
\item[$(iv)$]$({\mathcal L}_{\pi_1({\mathcal
J}A^h)}{\mathcal J}_{\varepsilon}\beta_s^{\varepsilon})_J(Z^h+U_r^{\varepsilon})=-({\mathcal
J}_{\varepsilon}\beta_s^{\varepsilon})(R(\pi_1(JA),Z)J).$
\end{enumerate}
\end{lemma}

\begin{proof}
By (\ref{[XhYh]}), if $X=\pi_1(A)$,
$$
({\mathcal
L}_{\pi_1(A^h)}\beta_s^{\varepsilon})_J(Z^h+U_r^{\varepsilon})=
-\beta_s^{\varepsilon}(R(X,Z)J)-\frac{1}{2}\beta_s^{\varepsilon}([X^h,\widetilde
a_{r}^{\varepsilon}]_J).
$$
By  Lemma~\ref{H-V brackets},
$$
\begin{array}{c}
[X^h,a_{2t-1}^{\varepsilon}]_J=(\widetilde{D_{X}a_{2t-1}^{\varepsilon}})_J=0,\\[6pt]
[X^h,a_{2t}^{\varepsilon}]_J=[X^h,{\mathcal
J}_{\varepsilon}\widetilde
a_{2t-1}^{\varepsilon}]_J=K_{\varepsilon}(\widetilde{D_{X}a_{2t-1}^{\varepsilon}})_J=0
\end{array}
$$
since $D a_{2t-1}^{\varepsilon}|_p=0$. This proves the first identity of the lemma. To prove the second one, we note that
if $f$ is a smooth function on ${\mathcal G}$ and $Y$ is a vector field on $M$, $ ({\mathcal
L}_{fY^h}\beta_s^{\varepsilon})_J(Z^h+U_r^{\varepsilon})=f({\mathcal
L}_{Y^h}\beta_s^{\varepsilon})_J(Z^h+U_r^{\varepsilon}) $ since $\beta_s^{\varepsilon}(Y^h)=0$. Now $(ii)$ follows from
(\ref{cal J/hor}) and the first identity of the lemma. Identities $(iii)$ and $(iv)$ are straightforward consequences
from $(i)$ and $(ii)$, respectively, since ${\mathcal
J}_{\varepsilon}\beta_{2t-1}^{\varepsilon}=\beta_{2t}^{\varepsilon}$, ${\mathcal
J}_{\varepsilon}\beta_{2t}^{\varepsilon}=-\beta_{2t-1}^{\varepsilon}$, $t=1,...,m^2-m$.
\end{proof}

\section{The Nijenhuis tensor}

\noindent {\it Notation}. We denote the Nijenhuis tensor of
${\mathcal J}_{\varepsilon}$ by $N_{\varepsilon}$,
$\varepsilon=1,2,3,4$.

 Moreover, given $J\in{\mathcal G}$ and $A,B\in T_{p}M\oplus T_{p}^{\ast}M$, $p=\pi(J)$, we define $1$-forms on
${\mathcal V}_J$  setting
$$
\omega^{\varepsilon}_{A,B}(W)=<(K_1W-K_{\varepsilon}W)(A),B>-<(K_1W-K_{\varepsilon}W)(B),A>,
\quad W\in{\mathcal V}_J.
$$

Also, let $S$ be a section of ${\cal G}$ in a neighbourhood of the
point $p=\pi(J)$  such that $S(p)=J$ and $D S|_p=0$  ($S$ being
considered as a section of $A(E')\oplus A(E'')$).

\vspace{0.1cm}
\begin{prop}\label {Nijenhuis}
Let $J\in {\cal G}$, $A,B\in T_{\pi(J)}M\oplus T_{\pi(J)}^{\ast}M$, $V,W\in {\cal V}_J$, $\varphi,\psi\in {\cal
V}_J^{\ast}$. Then, denoting the projection operators onto the horizontal and vertical components by ${\cal H}\oplus
{\cal H}^{\ast}$ and ${\mathcal V}\oplus {\mathcal V}^{\ast}$, we have:

\begin{enumerate}

\item[$(i)$] \ \\
$
({\cal H}\oplus {\cal H}^{\ast})N_{\varepsilon}(A^h,B^h)_J=(-[A,B]
+[SA,SB]-S[A,SB]-S[SA,B])^h_J.
$

\vspace{0.3cm}

\item[$(ii)$] \ \\
$$
\begin{array}{c}
({\mathcal V}\oplus {\mathcal V}^{\ast})N_{\varepsilon}(A^h,B^h)_J
= -R(\pi_1(A),\pi_1(B))J +R(\pi_1(JA),\pi_1(JB))J \\[8pt]

\hspace{4.5cm}-K_{\varepsilon}R(\pi_1(JA),\pi_1(B))J - K_{\varepsilon}R(\pi_1(A),\pi_1(JB))J \\[8pt]

-\omega^{\varepsilon}_{A,B}.
\end{array}
$$

\vspace{0.3cm}

\item[$(iii)$] \ \\
$N_{\varepsilon}(A^h,V)_J=(-(K_{\varepsilon}V)A+(K_1V)A)^{h}_J.$

\vspace{0.3cm}

\item[$(iv)$] \ \\
$N_{\varepsilon}(A^h,\varphi)_J\in {\cal H}_J\oplus {\cal H}_J^{\ast}$ ~ and \ \\
              \ \\
$
<\pi_{\ast}N_{\varepsilon}(A^h,\varphi)_J,B>=-\displaystyle{\frac{1}{2}}\varphi({\mathcal
V}N_{\varepsilon}(A^h,B^h)_J).$

\vspace{0.3cm}

\item[$(v)$] \ \\
$N_{\epsilon}(V+\varphi,W+\psi)_J=0$.
\end{enumerate}
\end{prop}
\begin{proof}
Formula $(i)$ follows from identity (\ref{[XhYh]}) and  Lemmas~\ref{brackets}, \ref{Lie deriv}, \ref{Half-diff}.
Also, the vertical part of $N_{\varepsilon}(A^h,B^h)_J$ is
equal to
$$
\begin{array}{c}
{\mathcal V}N_{\varepsilon}(A^h,B^h)_J
= -R(\pi_1(A),\pi_1(B))J +R(\pi_1(JA),\pi_1(JB))J \\[8pt]

\hspace{3.5cm}- {\mathcal J}_{\varepsilon}R(\pi_1(A),\pi_1(JB))J -{\mathcal J}_{\varepsilon}R(\pi_1(JA),\pi_1(B))J.
\end{array}
$$
The part of $N_{\varepsilon}(A^h,B^h)_J$ lying in ${\cal V}^{\ast}_J$ is the $1$-form whose value at every vertical
vector $W$ is
$$
\begin{array}{lll}
({\cal V}^{\ast}N_{\varepsilon}(A^h,B^h)_J)(W)= \\[4pt]
-\displaystyle{\frac{1}{2}}[\pi_2(JA)(\pi_1(WB))+\pi_2(WB)(\pi_1(JA)) \\[6pt]
\hspace{2.5cm}-\pi_2(B)(\pi_1((K_{\varepsilon}W)A))-\pi_2((K_{\varepsilon}W)A)(\pi_1(B))] \\
                                                                           \\
+\displaystyle{\frac{1}{2}}[\pi_2(JB)(\pi_1(WA))+\pi_2(WA)(\pi_1(JB)) \\[6pt]
\hspace{2.5cm}-\pi_2(A)(\pi_1((K_{\varepsilon}W)B))-\pi_2((K_{\varepsilon}W)B)(\pi_1(A))]\\\\[6pt]

=-[<JA,WB>-<(K_{\varepsilon}W)A,B>] +[<JB,WA>-<(K_{\varepsilon}W)B,A>].
\end{array}
$$
Note also that
$$
<JA,WB>=<JW(A),B>=<K_1W(A),B>.
$$
It follows that
$$
{\cal V}^{\ast}N_{\varepsilon}(A^h,B^h)_J=-\omega^{\varepsilon}_{A,B}.
$$
This proves $(ii)$.

  To prove $(iii)$ take a section $a$ of $A(M)$ near the point $p$ such that $a(p)=V$
and $\nabla a|_p=0$. Let $\widetilde a$ be the vertical vector field
defined by (\ref{eq tilde a}). Then it follows from Lemmas~\ref{H-V
brackets} and \ref{H-V Lie der&diff} that
$$
N_{\varepsilon}(A^h,V)_J=\frac{1}{2}N_{\varepsilon}(A^h,\widetilde a)_J=
(-(K_{\varepsilon}V)(A)+(J\circ V)A)^{h}_J.
$$

  To prove $(iv)$ let us take the vertical co-frame $\{\beta_{2t-1}^{\varepsilon},\beta_{2t}^{\varepsilon}\}$,
$t=1,...,m^2-m$, defined before the statement of Lemma~\ref{H-Vstar Lie der}. Set
$\varphi=\sum\limits_{s=1}^{2(m^2-m)}\varphi_s^{\varepsilon}\beta_s^{\varepsilon}$, $\varphi_s\in {\Bbb R}$. Let
$E_1,..., E_{2m}$  be a basis of $T_pM$ and $\xi_1,...,\xi_{2m}$ its dual basis. Then, by Lemma~\ref{H-Vstar Lie der}, we
have \vspace{0.1cm}
\begin{equation}\label{eq H-Vstar Nij}
\begin{array}{lll}
N_{\varepsilon}(A^h,\varphi)_J=\sum\limits_{s=1}^{2(m^2-m)}\varphi_s^{\varepsilon} N_{\varepsilon}(A^h,\beta_s^{\varepsilon})_J=\\
                                                                                   \\
\sum\limits_{s=1}^{2(m^2-m)}\sum\limits_{k=1}^{2m}\varphi_s^{\varepsilon}\{[\beta_s^{\varepsilon}(R(\pi_1(A),E_k)J)+
\beta_s^{\varepsilon}(K_{\varepsilon}R(\pi_1(JA),E_k)J)](\xi_k)^h_J \\
                                                                                    \\
+[\beta_s^{\varepsilon}(R(\pi_1(JA),E_k)J)-\beta_s^{\varepsilon}(K_{\varepsilon}R(\pi_1(A),E_k)J)](J\xi_k)^h_J\}.
\end{array}
\end{equation}
Moreover, note that
$$
<\xi_k,B>=\frac{1}{2}\xi_k(\pi_1(B)) \mbox { and }
<J\xi_k,B>=-<\xi_k,JB>=-\frac{1}{2}\xi_k(\pi_1(JB)).
$$
Therefore
$$
\sum_{k=1}^{2m}<\xi_k,B>E_k=\frac{1}{2}\pi_1(B) \mbox { and }
\sum_{k=1}^{2m}<J\xi_k,B>E_k=-\frac{1}{2}\pi_1(JB).
$$
Now $(iv)$ is an obvious consequence of (\ref{eq H-Vstar Nij}) and formula $(ii)$.

  Finally, identity $(v)$ follows from the fact that the generalized almost complex structure
${\cal J}_{\varepsilon}$ on every fibre of ${\cal G}$ is induced by a
complex structure.
\end{proof}

\section{Integrability conditions for generalized almost complex structures on generalized twistor spaces}

\begin{prop}\label{non-integr}
The generalized almost complex structures ${\mathcal J}_2, {\mathcal J}_3, {\mathcal J}_4$ are never integrable.
\end{prop}
\begin{proof}
Let $p\in M$ and take orthonormal bases $\{Q_1',...,Q_{2m}'\}$,
$\{Q_1'',...,Q_{2m}''\}$ of $E'_p$ and $E''_p$,  respectively. Let
$J'$ and $J''$ be the complex structures on $E'_p$ and $E''_p$ for
which $J'Q'_{2k-1}=Q'_{2k}$ and $J''Q''_{2k-1}=Q''_{2k}$,
$k=1,..,m$. Then $J=J'+J''$ is a generalized complex structure on
the vector space $T_pM$ compatible with the generalized metric
$E_p$. Define endomorphisms $S_{ij}'$ and $S_{ij}''$ by (\ref{eq
Sij}). Then $V'=S_{13}'+S_{42}'$ and $V''=S_{13}''+S_{42}''$ are
vertical tangent vectors of ${\mathcal G}$ at the point $J$. By
Proposition~\ref{Nijenhuis} $(iii)$,
$N_2(Q_1''^h,V'')=N_4(Q_1''^h,V'')=2Q_4''^h$,
$N_3(Q_1'^h,V')=2Q_4'^h$.
\end{proof}

\subsection{The case of the connection determined by a generalized metric}\label{main case}

Let $D'=\nabla^{E'}$ be the connection on $E'=E$ determined by the generalized metric $E$
(Proposition~\ref{connection}). The image of this connection under the isomorphism $pr_{TM}|E:E\to TM$
will be denoted by $\nabla$. The connection $\nabla$ has a skew-symmetric torsion $g(T(X,Y),Z)=d\Theta (X,Y,Z)$,
$X,Y,Z\in TM$.

We define a connection $D''$ on $E''$  transferring  the
connection $\nabla$ on $TM$ to $E''$ by means of the isomorphism
$pr_{TM}|E'':E''\to TM$. Since this isomorphism is an isometry with
respect to the metrics $g''=-<.\,,.>|E''$ and $g$, we get a
metric connection on $E''$. As in the preceding section,  define a connection $D$ on $TM\oplus T^{\ast}M$ setting
$D=D'$ on $E'=E$ and $D=D''$ on $E''$.

The connections induced by $\nabla$ on the bundles obtained from
$TM$ by algebraic operations like $T^{\ast}M$, $TM\oplus T^{\ast}M$,
etc. will also be denoted by $\nabla$.

\smallskip

Every section of $E'$ is of the form $S'=X+g(X)+\Theta(X)$ for a unique vector field $X$ and we have
\begin{equation}\label{D prime}
D_ZS'=\nabla_{Z}X+g(\nabla_{Z}X)+\Theta(\nabla_{Z}X),\quad Z\in TM,
\end{equation}
while
$$
\nabla_{Z}S'=\nabla_{Z}X+g(\nabla_{Z}X)+\Theta(\nabla_{Z}X)+(\nabla_{Z}\Theta)(X).
$$
Similarly for a section $S''=X-g(X)+\Theta(X)$ of $E''$
\begin{equation}\label{D double prime}
D_ZS''=\nabla_{Z}X-g(\nabla_{Z}X)+\Theta(\nabla_{Z}X),\quad Z\in TM,
\end{equation}
and
$$
\nabla_{Z}S'=\nabla_{Z}X-g(\nabla_{Z}X)+\Theta(\nabla_{Z}X)+(\nabla_{Z}\Theta)(X).
$$
Thus $\nabla$ preserves $E'$ or $E''$ if and only if $\nabla
\Theta=0$. Of course, this condition is not satisfied in general.
For example, if $\Theta$ is a closed $2$-form, which is not parallel
with respect to the Levi-Civita connection $\nabla^{LC}$, we have
$\nabla\Theta=\nabla^{LC}\Theta\neq 0$.

\smallskip

It follows from (\ref{pralpha}), (\ref{D prime}) and (\ref{D double
prime}) that if $\alpha$ is a one form on $M$ and $Z\in TM$,
\begin{equation}\label{Dalpha}
D_{Z}\alpha=D_{Z}\alpha_{E'}+D_{Z}\alpha_{E''}=g(\nabla_{Z}g^{-1}(\alpha)).
\end{equation}
Hence, $D_{Z}\alpha$ coincides with the covariant derivative of
$\alpha$ with respect to the connection $\nabla$ on $T^{\ast}M$:
\begin{equation}\label{Dalpha}
D\alpha=\nabla\alpha.
\end{equation}
On the other hand, if $X$ is a vector field on $M$ and $Z\in TM$, we
have by (\ref{prX}), (\ref{D prime}) and (\ref{D double prime})
$$
D_{Z}X=D_{Z}X_{E'}+D_{Z}X_{E''}=\nabla_{Z}X-g(\nabla_{Z}\{(g^{-1}\circ\Theta)(X)\})+\Theta(\nabla_{Z}X).
$$
Moreover, for every vector field $Y$,
$$
\begin{array}{c}
g(\nabla_{Z}\{(g^{-1}\circ\Theta)(X)\})(Y)=Z(g((g^{-1}\circ\Theta)(X),Y))-g((g^{-1}\circ\Theta)(X),\nabla_{Z}Y)\\[8pt]
=Z(\Theta(X,Y))-\Theta(X,\nabla_{Z}Y)=(\nabla_{Z}\Theta)(X,Y)+\Theta(\nabla_{Z}X,Y).
\end{array}
$$
Thus
\begin{equation}\label{DX}
D_{Z}X=\nabla_{Z}X-(\nabla_{Z}\Theta)(X).
\end{equation}
Therefore the connection $D$ does not preserves $TM$ in general. In
particular, the connection $D$ on $TM\oplus T^{\ast}M$ is different
from the connection on this bundle induced by $\nabla$; the two
connections coincide if and only if $\nabla\Theta=0$.

\smallskip

Denote by ${\mathcal Z}={\mathcal Z}(TM,g)$ the bundle over $M$
whose fibre at a point $p\in M$ consists of complex structures on
$T_pM$ compatible with the metric $g$ (the usual twistor space of
$(M,g)$). Consider ${\mathcal Z}$ as a submanifold of the bundle
$A(TM)$ of $g$-skew-symmetric endomorphisms of $TM$. The projections
$pr_{TM}|E':E'\to TM$ and $pr_{TM}|E'':E''\to TM$ yield an isometric
bundle-isomorphism $A(E')\oplus A(E'')\to A(TM)\oplus A(TM)$ sending
the connection $D'\oplus D''$ to the connection $\nabla\oplus
\nabla$. The restriction of this map to ${\mathcal Z}(E')\times
{\mathcal Z}(E'')\cong {\mathcal G}$ yields an isomorphism of
${\mathcal G}$ onto ${\mathcal Z}\times {\mathcal Z}$ given by
${\mathcal G}\ni J\to (J_1,J_2)$, where $J_1=(pr_{T}|E')\circ J
\circ (pr_{T}|E')^{-1}$, $J_2=(pr_{T}|E'')\circ J \circ
(pr_{T}|E'')^{-1}$. In the case when $M$ is oriented it identifies
the connected components of ${\mathcal G}$ with the four product
bundles ${\mathcal Z}_{\pm}\times {\mathcal Z}_{\pm}$, ${\mathcal
Z}_{\pm}$ being the bundle over $M$ whose sections are the almost
complex structures on $M$ compatible with the metric and $\pm$ the
orientation.

\begin{prop}\label {hor com zero}
$({\cal H}\oplus {\cal H}^{\ast})N_{\varepsilon}(A^h,B^h)_J=0$
for every $J\in {\mathcal G}$ and every $A,B\in TM\oplus T^{\ast}M$ if and
only if $d\Theta=0$.
\end{prop}
\begin{proof}
  Let $J\in {\mathcal G}$ and let $S=(S',S'')$ be a section of ${\cal G}$ in a neighbourhood of the point $p=\pi(J)$
with the properties that $S(p)=J$ and $D S|_p=0$.

According to Proposition ~\ref {Nijenhuis} $(i)$,  $({\cal H}\oplus {\cal H}^{\ast})N_{\varepsilon}(A^h,B^h)_J=0$
if and only if the Nijenhuis tensor $N_S$ of the generalized almost
complex structure $S$ on $M$ vanishes at the  point $p$. Let $S_1$
and $S_2$ be the almost complex structures on $M$ determined by $S$,
$$
S_1=(\pi_1|E')\circ S'\circ (\pi_1|E')^{-1},\quad
S_2=(\pi_1|E'')\circ S''\circ (\pi_1|E'')^{-1}.
$$
These structures are compatible with the metric $g$ and we denote
their fundamental $2$-forms  by $\Omega_1$ and $\Omega_2$,
respectively:
$$
\Omega_1(X,Y)=g(X,S_1Y),\quad \Omega_2(X,Y)=g(X,S_2Y),\quad X,Y\in
TM.
$$

 Denote by $K$ the generalized complex structure on $M$ with the
block-matrix
$$K=\frac{1}{2}
  \left(
  \begin{array}{cc}
    S_{1}+S_{2} & \Omega_{1}^{-1}-\Omega_{2}^{-1} \\
    -(\Omega_{1}-\Omega_{2})& -(S_{1}^{\ast}+S_{2}^{\ast}) \\
  \end{array}
 \right).
$$
By Proposition~\ref{J-J1,2}, the generalized complex structure $S$ is the
$B$-transform of $K$ by means of the form $\Theta$:
$$
S=e^{\Theta}Ke^{-\Theta}.
$$
Let $N_K$ be the Nijensuis tensor of the generalized almost complex
structure $K$. Set
\begin{equation}\label{notat}
A=X+\alpha,\quad B=Y+\beta,\quad
KA=\widehat{X}+\widehat{\alpha},\quad
KB=\widehat{Y}+\widehat{\beta},
\end{equation}
where $X,Y,\widehat{X},\widehat{Y}\in TM$ and
$\alpha,\beta,\widehat{\alpha},\widehat{\beta}\in T^{\ast}M$. Then,
by Proposition~\ref{Courant-B-transf} and the fact that $e^{-\Theta}|T^{\ast}M=Id$,
$$
\begin{array}{c}
N_S(e^{\Theta}A,e^{\Theta}B)=e^{\Theta}N_K(A,B)-\imath_{Y}\imath_{X}d\Theta
+\imath_{\widehat{Y}}\imath_{\widehat{X}}d\Theta\\[8pt]
\hspace{7cm}-e^{\Theta}K(\imath_{Y}\imath_{\widehat{X}}d\Theta
+\imath_{\widehat{Y}}\imath_{X}d\Theta).
\end{array}
$$
It follows that $({\cal H}\oplus {\cal
H}^{\ast})N_{\varepsilon}(A^h,B^h)_J=0$ for every $A,B\in TM\oplus
T^{\ast}M$ if and only if at the point $p=\pi(J)$
\begin{equation}\label{cond}
N_K(A,B)=\imath_{Y}\imath_{X}d\Theta
-\imath_{\widehat{Y}}\imath_{\widehat{X}}d\Theta
+K(\imath_{Y}\imath_{\widehat{X}}d\Theta+\imath_{\widehat{Y}}\imath_{X}d\Theta).
\end{equation}
 We have
$$
\nabla S_1=(\pi_1|E')\circ (D S')\circ (\pi_1|E')^{-1}
$$
since the connection $\nabla$ on $TM$ is obtained from the
connection $D|E'=\nabla^{E'}$ by means of the isomorphism
$\pi_1|E':E'\to TM$. In particular $\nabla S_1|_p=0$. Similarly,
$\nabla S_2|_p=0$. Then $\nabla S_k^{\ast}|_p=0$, $k=1,2$, and
$\nabla\Omega_k|_p=-\nabla(g\circ S_k)|_p=0$,
$\nabla\Omega_k^{-1}=\nabla(S_k\circ g^{-1})|_p=0$ since $\nabla
g=\nabla g^{-1}=0$. It follows that $\nabla K|_p=0$. Extend $X$ and
$\alpha$ to a vector field $X$ and a $1$-form $\alpha$ on $M$ such
that $\nabla X|_p=0$ and $\nabla\alpha|_p=0$; similarly for $Y$ and
$\beta$. In this way we obtain sections $A=X+\alpha$ and $B=Y+\beta$
of $TM\oplus T^{\ast}M$ such that  $\nabla A|_p=\nabla B|_p=0$ and
$\nabla KA|_p=\nabla KB|_p=0$.

In order to
compute $N_K(A,B)$ we need the following simple observation:
Let $Z$ be a vector field and $\omega$ a $1$-form on $M$ such that
$\nabla Z|_p=0$ and $\nabla\omega|_p=0$. Then, for every $Z'\in
T_pM$,
$$
\begin{array}{c}
({\cal L}_{Z}\omega)(Z')_p
=(\nabla_{Z}\omega)(Z')_p+\omega(T(Z,Z')=\omega(T(Z,Z')),\\[8pt]
(d\,\imath_{Z}\omega)(Z')_p=Z'_p(\omega(Z))=(\nabla_{Z'}\omega)(Z)_p=0,
\end{array}
$$
where $T(Z,Z')$ is the torsion tensor of the connection $\nabla$.
For $Z\in TM$, let $\imath_{Z}T:TM\to TM$ be the map $Z'\to
T(Z,Z')$. Then, under the notation in (\ref{notat}), we have
\begin{equation}\label{NK}
\begin{array}{c}
N_K(A,B)=T(X,Y)-T(\widehat{X},\widehat{Y})+\alpha(\imath_{Y}T)-\beta(\imath_{X}T)
-\widehat{\alpha}(\imath_{\widehat{Y}}T)+\widehat{\beta}(\imath_{\widehat{X}}T)\\[8pt]
+K[T(\widehat{X},Y)+T(X,\widehat{Y})+\widehat{\alpha}(\imath_{Y}T)-\beta(\imath_{\widehat{X}}T)
+\alpha(\imath_{\widehat{Y}}T)-\widehat{\beta}(\imath_{X}T)].
\end{array}
\end{equation}

If $\alpha=g(X')$ for some (unique)  $X'\in TM$, we have
$$
\alpha(\imath_{Y}T)=g(T(X',Y))=\imath_{Y}\imath_{X'}d\Theta,\quad
Y\in TM,
$$
and
$$
K(\alpha\circ\imath_{Y}T)=\displaystyle{\frac{1}{2}}[(\Omega_1^{-1}-\Omega_2^{-1})\imath_{Y}\imath_{X'}d\Theta
-(S_1^{\ast}+S_2^{\ast})\imath_{Y}\imath_{X'}d\Theta]
$$

Moreover, $g(T(X,Y))=\imath_{Y}\imath_{X}d\Theta$ for every $X,Y\in TM$, hence
$$
K(T(X,Y))=\displaystyle{\frac{1}{2}}[(\Omega_1^{-1}+\Omega_2^{-1})\imath_{Y}\imath_{X}d\Theta
-(S_1^{\ast}-S_2^{\ast})\imath_{Y}\imath_{X}d\Theta].
$$
Note also that
$$
\begin{array}{c}
\widehat{X}=\displaystyle{\frac{1}{2}}[S_1(X+X')+S_2(X-X')],\\[8pt]
\widehat{\alpha}=\displaystyle{\frac{1}{2}}[g(S_1(X+X')-S_2(X-X'))].
\end{array}
$$

Now suppose that
$$
({\cal H}\oplus {\cal H}^{\ast})N_{\varepsilon}(A^h,B^h)_J=0, \quad
A,B\in TM\oplus T^{\ast}M.
$$
Then, by (\ref{cond}),
\begin{equation}\label{cond-2-forms}
N_K(g(X),g(Y))=\displaystyle{\frac{1}{4}\imath_{(S_1X-S_2X)}\imath_{(S_1Y-S_2Y)}d\Theta}.
\end{equation}
Therefore the tangential component of $N_K(g(X),g(Y))$ vanishes. Hence by (\ref{NK})
$$
\begin{array}{c}
-T(S_1X-S_2X,S_1Y-S_2Y)\\[6pt]
-(\Omega^{-1}_1-\Omega^{-1}_2)\imath_{(S_1X-S_2X)}\imath_{Y}d\Theta
+(\Omega^{-1}_1-\Omega^{-1}_2)\imath_{(S_1Y-S_2Y)}\imath_{X}d\Theta=0.
\end{array}
$$
Applying the map $g$ to both sides of the latter identity we obtain by means of the identities
$g\circ\Omega_k^{-1}=-S_k^{\ast}$ that for every $X,Y,Z\in T_pM$
\begin{equation}\label{ff1}
\begin{array}{c}
d\Theta(S_1X-S_2X,Y,S_1Z-S_2Z)+d\Theta(X,S_1Y-S_2Y,S_1Z-S_2Z)\\[8pt]
=-d\Theta(S_1X-S_2X,S_1Y-S_2Y,Z).
\end{array}
\end{equation}
Applying (\ref{ff1}) for the generalized almost complex structure
determined by the complex structures $(-S_1,S_2)$ on $T_pM$ and
comparing the obtained identity with (\ref{ff1}) we see that
\begin{equation}\label{dTheta-1}
\begin{array}{c}
d\Theta(S_1X,Y,S_1Z)+d\Theta(S_2X,Y,S_2Z)+d\Theta(X,S_1Y,S_1Z)+d\Theta(X,S_2Y,S_2Z)\\[8pt]
=-d\Theta(S_1X,S_1Y,Z)-d\Theta(S_2X,S_2Y,Z).
\end{array}
\end{equation}
Computing the co-tangential component of $N_K(g(X),g(Y))$ by means of (\ref{NK}), then applying identity
(\ref{cond-2-forms}) for the generalized almost complex structures determined by  $(S_1,S_2)$ and $(-S_1,S_2)$, we obtain
\begin{equation}\label{dTheta-2}
\begin{array}{c}
-d\Theta(S_1X,Y,S_1Z)+d\Theta(S_2X,Y,S_2Z)-d\Theta(X,S_1Y,S_1Z)+d\Theta(X,S_2Y,S_2Z)\\[8pt]
=d\Theta(S_1X,S_1Y,Z)-3d\Theta(S_2X,S_2Y,Z).
\end{array}
\end{equation}
It follows from (\ref{dTheta-1}) and (\ref{dTheta-2}) that
$$
d\Theta(S_2X,Y,S_2Z)+d\Theta(X,S_2Y,S_2Z)=-2d\Theta(S_2X,S_2Y,Z).
$$
Hence
$$
2d\Theta(X,Y,Z)=d\Theta(X,S_2Y,S_2Z)+d\Theta(S_2X,Y,S_2Z).
$$
The latter identity holds if and only if it holds for every
$X,Y,Z\in T_pM$ with $|X|=|Y|=1$, $X\perp Y$. Given three tangent
vectors with these properties, there exists a complex structure
$S_2$ on $T_pM$ such that $Y=S_2X$. It follows that
$$
d\Theta(X,Y,Z)=0,\quad X,Y,Z\in T_pM.
$$
Conversely, if $d\Theta=0$, then $T=0$ and we have $N_K=0$ by (\ref{NK}).  Thus the condition (\ref{cond})
is trivially satisfied. Therefore
$$
({\cal H}\oplus {\cal H}^{\ast})N_{\varepsilon}(A^h,B^h)_J=0, \quad
A,B\in TM\oplus T^{\ast}M.
$$
\end{proof}

\smallskip

Suppose that $M$ is oriented and $dim\,M=4k$. Then the above proof
still holds true if we, instead of ${\mathcal G}$, consider a
connected component of it. Indeed, the almost complex structures
$S_1$ and $-S_1$ induce the same orientation and, moreover, the
complex structure $S_2$ with the property $Y=S_2X$ used at the end
of the proof can be chosen to induce the given or the opposite
orientation of $M$. Thus we have the following.

\begin{prop}\label{dim=4k} If $M$ is oriented and $dim\,M=4k$, then
$$({\cal H}\oplus {\cal H}^{\ast})N_{\varepsilon}(A^h,B^h)_J=0$$
for every $J$ in a connected component of ${\mathcal G}$ and every $A,B\in TM\oplus T^{\ast}M$ if and
only if $d\Theta=0$.
\end{prop}

Considering the double orientable covering of $M$, if necessary, we
may assume that $M$ itself is orientable. Fix an orientation on $M$.
Denote by ${\mathcal G}_{++}$ the subbundle of ${\mathcal G}$ whose
fibre at a point $p\in M$ consists of generalized complex structures
$J$ on $T_pM$ compatible with the generalized metric $E_p$ and such
that the complex structures $J_1$ and $J_2$ on $T_pM$ determined by
$J$ via (\ref{J1,2}) induce the orientation of $T_pM$. We define
subbundles ${\mathcal G}_{--}$, ${\mathcal G}_{+-}$, ${\mathcal
G}_{-+}$ in a similar way. These are the connected components of the
space ${\mathcal G}$.

\smallskip

\noindent {\it  Convention}.  Henceforth we assume that $M$ is
oriented and of dimension $4k$.

\smallskip

Recall that if $R$ is the curvature tensor of the Levi-Civita connection of
$(M,g)$,  the
curvature operator ${\mathcal R}$ is the self-adjoint endomorphism of
$\Lambda ^2TM$ defined by
$$
g({\mathcal R}(X\land Y),Z\land T)=g(R(X,Y)Z,T),\quad X,Y,Z,T\in TM.
$$
The metric on $\Lambda^2TM$ used in the left-hand side of the latter identity is defined by
$$
g(X_1\wedge X_2,X_3\wedge X_4)=g(X_1,X_3)g(X_2,X_4)-g(X_1,X_4)g(X_3,X_4).
$$
As is well-known, the curvature operator decomposes as (see, for example, \cite [Section 1 G, H]{Besse})
\begin{equation}\label{curv decom}
{\mathcal R}=\frac{s}{n(n-1)}Id + {\mathcal B}+ {\mathcal W},
\end{equation}
where $s$ is the scalar curvature of the manifold $(M,g)$ and ${\mathcal B}$, ${\mathcal W}$ correspond to its traceless
Ricci tensor and Weyl conformal tensor, respectively. If $\rho:TM\to TM$ is the Ricci operator,
$g(\rho(X),Y)=Ricci(X,Y)$, the operator ${\mathcal B}$ is given by
\begin{equation}\label{B}
{\mathcal B}(X\wedge Y)=\frac{1}{n-2}[\rho(X)\wedge Y + X\wedge\rho(Y) -\frac{2s}{n} X\wedge Y],
\quad X,Y\in TM.
\end{equation}
Thus, a Riemannian manifold  is Einstein exactly when ${\cal B}=0$;
it is conformally flat when ${\mathcal W}=0$.

\smallskip

If the dimension of $M$ is four, the Hodge star operator
defines an involution $\ast$ of $\Lambda^2TM$ and
we have the orthogonal decomposition
$$
\Lambda^2TM=\Lambda^2_{-}TM\oplus\Lambda^2_{+}TM,
$$
where $\Lambda^2_{\pm}TM$ are the subbundles of $\Lambda^2TM$
corresponding to the $(\pm 1)$-eigenvalues of the operator $\ast$.
Accordingly, the operator ${\mathcal W}$ has an extra decomposition
${\mathcal W}={\mathcal W}_{+}+{\mathcal W}_{-}$ where ${\mathcal
W}_{\pm}={\mathcal W}$ on $\Lambda^2_{\pm}TM$ and ${\mathcal
W}_{\pm}=0$ on $\Lambda^2_{\mp}TM$. The operator ${\mathcal B}$ does
not have a decomposition of this type since it maps
$\Lambda^2_{\pm}TM$ into $\Lambda^2_{\mp}TM$.

Recall also that a Riemannian manifold $(M,g)$ is  called self-dual
(anti-self-dual), if ${\cal W}_{-}=0$ (resp. ${\cal W}_{+}=0$).

\smallskip

According to Propositions~\ref{Nijenhuis} and \ref{dim=4k}, the
restriction of the generalized almost complex structure ${\mathcal
J}_1$ to a connected component $\widetilde{\mathcal G}$ of
${\mathcal G}$ is integrable if and only if $d\Theta=0$ and for
every $p\in M$, $A,B\in T_pM$, and for every generalized complex
structure $J\in\widetilde{\mathcal G}$ on $T_pM$
$$
\begin{array}{c}
-R(\pi_1(A),\pi_1(B))J +R(\pi_1(JA),\pi_1(JB))J \\[8pt]
-K_1R(\pi_1(JA),\pi_1(B))J - K_1R(\pi_1(A),\pi_1(JB))J =0,
\end{array}
$$
where $R$ is the curvature tensor of the connection $D$ on the
bundle $A(E')\oplus A(E'')$. If $(J_1,J_2)$ are the complex
structures on $T_pM$ determined by $J$, the latter identity is
equivalent to the identities
\begin{equation}\label{I and II comp}
\begin{array}{c}
-R(\pi_1(A),\pi_1(B))J_r +R(\pi_1(JA),\pi_1(JB))J_r \\[8pt]
-J_r\circ R(\pi_1(JA),\pi_1(B))J_r- J_r\circ R(\pi_1(A),\pi_1(JB))J_r =0, \quad r=1,2,
\end{array}
\end{equation}
where $R$ is the curvature tensor on the bundle $A(TM)$ of skew-symmetric endomorphism of $TM$ induced by the
connection $\nabla$.

Assume that $d\Theta=0$. Then $\nabla$ is the Levi-Civita connection
of the Riemannian manifold $(M,g)$. Every $A\in E_p'$ is of the form
$A=X+g(X)+\Theta(X)$ for some (unique) $X\in T_pM$ and
$JA=J_1X+g(J_1X)+\Theta(J_1X)$. Similarly, if $B\in E''_p$, then
$B=Y-g(Y)+\Theta(Y)$, $Y\in T_pM$ and
$JB=J_2Y-g(J_2Y)+\Theta(J_2Y)$.  It follows that the identity
(\ref{I and II comp}) is equivalent to the condition that for every
$X,Y,Z,U\in T_pM$ and every complex structures $(J_1,J_2)$ on $T_pM$
corresponding to a generalized complex structure $J$ in
$\widetilde{\mathcal G}$
\begin{equation}\label{Jklr}
\begin{array}{c}
g({\mathcal R}(X\wedge Y-J_jX\wedge J_lY), Z\wedge U-J_rZ\wedge J_rU)\\[6pt]
\hspace{3.3cm}=g({\mathcal R}(J_jX\wedge Y+X\wedge J_lY),J_rZ\wedge U+Z\wedge J_rU),\\[6pt]
 j,l,r=1,2.
\end{array}
\end{equation}
The complex structures $(J_1,J_2)$ in the latter identities are
compatible with the metric $g$ and, moreover, they induce the
orientation of  $T_pM$ if we consider the connected component
$\widetilde{\mathcal G}={\mathcal G}_{++}$, while $(J_1,J_2)$ induce
the opposite orientation in the case $\widetilde{\mathcal
G}={\mathcal G}_{--}$. If $\widetilde{\mathcal G}={\mathcal
G}_{+-}$, the complex structure $J_1$ induces the given orientation
of $T_pM$ and $J_2$ yields the opposite one, and vice versa if
$\widetilde{\mathcal G}={\mathcal G}_{-+}$.

\smallskip
For $j=l=r$ identity (\ref{Jklr}) coincides with the integrability
condition for the Atiyah-Hitchin-Singer almost complex structure
\cite{AtHiSi} on the positive or negative twistor space of $(M,g)$,
the fibre bundles over $M$ whose fibre at every point $p\in M$
consists of the complex structures on $T_pM$ compatible with the
metric and $\pm$ the orientation of $T_pM$ (see, for example,
\cite[Section 5.19]{Will}). It is also well known that this
integrability condition is equivalent to $(M,g)$ being conformally
flat if $dim\,M\geq 6$. If $dim\,M=4$ the integrability condition is
equivalent to anti-self-duality of $(M,g)$ in the case of positive
twistor spaces and to its self-duality  when considering the
negative twistor space.

\begin{tm} $I$. Suppose that $dim\,M=4$.

$(a)$ The restriction of the generalized complex structure ${\mathcal J}_1$ to ${\mathcal G}_{++}$  is integrable
if and only if $(M,g)$ is anti-self-dual and Ricci flat.

$(b)$ The restriction ${\mathcal J}_{1}|{\mathcal G}_{--}$ is integrable if and only if $(M,g)$ is self-dual and
Ricci flat.

II. If $dim\,M=4k\geq 6$, each of the restrictions of ${\mathcal
J}_1$ to ${\mathcal G}_{++}$ and  ${\mathcal G}_{--}$ is integrable
if and only if the manifold $(M,g)$ is flat.

\end{tm}

\begin{proof} Let $E_1,..., E_{n}$ be an oriented orthonormal basis of a tangent space $T_pM$. It is convenient to set
$E_{ab}=E_a\wedge E_b$ and $\rho_{ab}=Ricci(E_a,E_b)$, $a,b=1,...,n$.

Suppose that the structure ${\mathcal J}_1| {\mathcal G}_{++}$ is integrable. Let $J_1$ and $J_2$ be  complex structures
on $T_pM$ for which $J_1E_{1}=E_{3}$, $J_1E_{2}=-E_{4}$ and $J_2E_{1}=E_{4}$, $J_2E_{2}=E_{3}$. Identity (\ref{Jklr})
with $j=l=1$, $r=2$, and $(X,Y,Z,U)=(E_1,E_2,E_3,E_4)$ gives
\begin{equation}\label{k=l=1,r=2}
g({\mathcal R}(E_{12}+E_{34}),E_{12}+E_{34})+g({\mathcal R}(E_{14}+E_{23}),E_{13}+E_{42})=0.
\end{equation}
If $dim\,M=4$, then
$E_{12}+E_{34},E_{14}+E_{23}\in\Lambda^2_{+}T_pM$ and ${\mathcal
W}_{+}=0$, hence
$$
{\mathcal W}(E_{12}+E_{34})=
{\mathcal W}(E_{14}+E_{23})=0.
$$
If  $dim\,M\geq 6$, we have ${\mathcal W}=0$. Thus,  in both cases
by (\ref{curv decom})
$$
\begin{array}{c}
g({\mathcal R}(E_{12}+E_{34}),E_{12}+E_{34})+g({\mathcal
R}(E_{14}+E_{23}),E_{13}+E_{42})\\[6pt]
=\displaystyle{\frac{2s}{n(n-1)}}+g({\mathcal
B}(E_{12}+E_{34}),E_{12}+E_{34})+g({\mathcal
B}(E_{14}+E_{23}),E_{13}+E_{42}).
\end{array}
$$
By (\ref{B})
$$
\begin{array}{c}
g({\mathcal
B}(E_{12}+E_{34}),E_{12}+E_{34})=\displaystyle{\frac{1}{n-2}}[\rho_{11}+\rho_{22}+\rho_{33}+\rho_{44}-\frac{4s}{n}],\\[8pt]
g({\mathcal B}(E_{14}+E_{23}),E_{13}+E_{42})=0.
\end{array}
$$
Then by (\ref{k=l=1,r=2})
$$
\frac{2s}{n(n-1)}+\frac{1}{n-2}[\rho_{11}+\rho_{22}+\rho_{33}+\rho_{44}-\frac{4s}{n}]=0.
$$
In a similar way we see that
$$
\frac{2s}{n(n-1)}+\frac{1}{n-2}[\rho_{4i-3,4i-3}+\rho_{4i-2,4i-2}+\rho_{4i-1,4i-1}+\rho_{4i,4i}-\frac{4s}{n}]=0
$$
for $i=1,2,...,k$. Summing up these identities we get $s=0$.

In order to show that ${\mathcal B}=0$ we apply identity
(\ref{Jklr}) with $j=1$, $l=2$ and  take $J_1$, $J_2$ to be the
complex structures introduced above. Subtracting the identities
corresponding to $X=Y=E_2$ and $X=Y=E_3$, we get
$$
g({\mathcal R}(E_{13}-E_{42}),Z\wedge U-J_rZ\wedge
J_rU)=-g({\mathcal R}(E_{12}-E_{34}),J_rZ\wedge U+Z\wedge J_rU).
$$
Subtraction of the identities corresponding to $X=Y=E_1$ and
$X=Y=E_4$ gives
$$
g({\mathcal R}(E_{13}-E_{42}),Z\wedge U-J_rZ\wedge J_rU)=g({\mathcal
R}(E_{12}-E_{34}),J_rZ\wedge U+Z\wedge J_rU).
$$
Thus
\begin{equation}\label{aux}
g({\mathcal R}(E_{13}-E_{42}),Z\wedge U-J_rZ\wedge
J_rU)=0=g({\mathcal R}(E_{12}-E_{34}),J_rZ\wedge U+Z\wedge J_rU).
\end{equation}
If $dim\,M=4$, every $2$-vector of the form $Z\wedge U-J_rZ\wedge
J_rU$ lies in $\Lambda^{2}_{+}T_pM$ since $J_r$ is compatible with
the metric and  orientation of $T_pM$. Therefore ${\mathcal
W}(Z\wedge U-J_rZ\wedge J_rU)={\mathcal W}( J_rZ\wedge U+Z\wedge
J_rU)=0$. If $dim\,M\geq 6$, this is obvious. Then the first
identity in (\ref{aux}) with $r=1$ and $(Z,U)=(E_1,E_3)$ gives
$$
g({\mathcal B}(E_{13}-E_{42}),E_{13}+E_{42})=0.
$$
It follows by (\ref {B}) that
$$
\rho_{11}-\rho_{22}+\rho_{33}-\rho_{44}=0.
$$
Applying the latter identity for the basis
$E_1,E_3,E_4,E_2,E_5,...,E_n$, we get
$$
\rho_{11}-\rho_{22}-\rho_{33}+\rho_{44}=0.
$$
Therefore $\rho_{11}=\rho_{22}$. It follows that
$\rho_{11}=\rho_{aa}$ for $a=1,...,n$. This implies $\rho_{aa}=0$
for every $a=1,...,n$ since the scalar curvature vanishes. Moreover,
the first identity in (\ref{aux}) for $r=2$ and $(Z,U)=(E_1,E_2)$
reads as
$$
g({\mathcal B}(E_{13}-E_{42}),E_{12}+E_{34})=0.
$$
This gives $- \rho_{14}+\rho_{23}=0$. Similarly, it follows from the
second identity in (\ref{aux}) with $r=1$ and $(Z,U)=(E_2,E_2)$ that
$\rho_{14}+\rho_{23}=0$. Hence $\rho_{14}=\rho_{23}=0$. It follows
that $\rho_{ab}=0$, $a\neq b$. Therefore $Ricci=0$.

Conversely, it is obvious that identity (\ref{Jklr}) is satisfied if
$(M,g)$ is flat. In the case when $s=0$ and ${\mathcal B}={\mathcal
W}_{+}=0$, identity (\ref{Jklr}) is also trivially satisfied since
for every $X,Y\in T_pM$ and every complex structure $J$ on $T_pM$
compatible with the metric and orientation, the $2$-vector $X\wedge
Y- JX\wedge JY$ lies in $\Lambda^2_{+}T_pM$, so ${\mathcal
R}(X\wedge Y- JX\wedge JY)=0$.

This proves statements $ I\, (a)$ and $II$. Statement $ I\, (b)$ is an obvious corollary of $ I\, (a)$ by
reversing the orientation of $M$.

\end{proof}

\smallskip

\noindent{\bf Remark} 6. By a result of Hitchin \cite{Hit74} if $M$
is a compact anti-self-dual, Ricci flat four-dimensional manifold,
then either $M$ is flat or is a $K3$-surface, an Enriques surface or
the quotient of an Enriques surface by a free anti-holomorphic
involution.

\smallskip

\begin{tm} Each of the restrictions ${\mathcal J}_1|{\mathcal G}_{+-}$ and ${\mathcal J}_1|{\mathcal G}_{-+}$ is an
integrable generalized almost complex structure if and only $(M,g)$ is of constant sectional curvature.
\end{tm}

\begin{proof} Suppose that ${\mathcal J}_1|{\mathcal G}_{+-}$ is integrable. Then, by the preceding remarks,
if $dim\,M=4$ $(M,g)$ is both anti-self-dual and self-dual, hence
${\mathcal W}=0$; if $dim\,M\geq 6$, we also have ${\mathcal W}=0$.
Take an orthonormal oriented basis $E_1,...,E_n$ of a tangent space
$T_pM$ and  consider (\ref{Jklr}) with $j=1$, $l=r=2$. Take for
$J_1$ and  $J_2$  the complex structures on $T_pM$ for which
$J_1E_{1}=E_{3}$, $J_1E_{4}=E_{2}$ and $J_2E_{1}=E_{4}$,
$J_2E_{2}=-E_{3}$. Adding the identities corresponding to
$(X,Y)=(E_1,E_2)$ and $(X,Y)=(E_3,E_4)$, we get
$$
g({\mathcal R}(E_{12}+E_{34}),Z\wedge U-J_2Z\wedge J_2U)+g({\mathcal
R}(E_{14}+E_{23}),J_2Z\wedge U + Z\wedge J_2U)=0.
$$
For $(Z,U)=(E_1,E_2)$, this gives
$$
g({\mathcal R}(E_{12}+E_{34}),E_{12}-E_{34})-g({\mathcal
R}(E_{14}+E_{23}),E_{13}-E_{42})=0.
$$
Then, since ${\mathcal W}=0$, we obtain by means of (\ref{B})
$$
\rho_{11}+\rho_{22}-\rho_{33}-\rho_{44}-2(\rho_{12}+\rho_{34})=0.
$$
Applying this identity for the basis
$(-E_1,E_2,-E_3,E_4,E_5,...,E_n)$ we have
$$
\rho_{11}+\rho_{22}-\rho_{33}+\rho_{44}+2(\rho_{12}+\rho_{34})=0.
$$
Hence
$$
\rho_{11}+\rho_{22}=\rho_{33}+\rho_{44},\quad \rho_{12}=-\rho_{34}.
$$
The first of these identities imply
$$
\rho_{11}+\rho_{22}=\rho_{aa}+\rho_{bb}\hspace{0.3 cm} \rm{for}
\hspace{0.3 cm}  a\neq b,\quad a,b=1,...,n.
$$
It follows that
$$
\rho_{aa}+\rho_{bb}=\frac{s}{2k}, \quad a\neq b
$$
Applying the  above obtained identity  $\rho_{12}=-\rho_{34}$ for
the basis $(E_2,E_1,-E_3,\\ E_4,E_5,...,E_n)$ we get
$\rho_{12}=\rho_{34}$, thus $\rho_{12}=\rho_{34}=0$.  It follows
that
$$
\rho_{ab}=0,\quad a\neq b.
$$
Now we note that the condition ${\mathcal B}=0$ is equivalent to
$$
\rho_{aa}+\rho_{bb}-\frac{2s}{n}=0, \quad \rho_{ab}=0,\quad a\neq
b,\quad a,b=1,...,n.
$$
Thus we can conclude that ${\mathcal B}=0$.  Therefore ${\mathcal
R}=\displaystyle{\frac{s}{n(n-1)}}Id$, i.e. $(M,g)$ is of constant
sectional curvature.

Conversely, if ${\mathcal R}=\displaystyle{\frac{s}{n(n-1)}}Id$, a
straightforward computation shows that identity (\ref{Jklr}) is
satisfied.

\end{proof}

\section{Natural isomorphisms of generalized twistor spaces}\label{aut}

\noindent {\bf I}. Let $f:A(E')\oplus A(E'')\to A(E')\oplus A(E'')$ be the bundle isomorphism $a=(a',a'')\to (a',-a'')$.
The differential of this isomorphism preserves the horizontal lifts, $f_{\ast}X^h_{a}=X^h_{f(a)}$, and if $V=(V',V'')$ is
a vertical vector, $f_{\ast}V=(V',-V'')$. The restriction of $f$ to the generalized twistor space ${\mathcal G}$ is an
automorphism of ${\mathcal G}$. The automorphism $F=f_{\ast}\oplus (f^{-1})^{\ast}$ of $T{\mathcal G}\oplus
T^{\ast}{\mathcal G}$ preserves the horizontal and vertical subbundles and sends the generalized almost complex structure
${\mathcal J}_{\varepsilon}$ to the structure $\bar{\mathcal J}_{\varepsilon}$ given by $\bar{\mathcal
J}_{\varepsilon}A^h_J=(F^{-1}(J)A)^h_J$ for $A\in T_{\pi(J)}M\oplus T^{\ast}_{\pi(J)}M$ and $\bar{\mathcal
J}_{\varepsilon}={\mathcal J}_{\varepsilon}$ on ${\mathcal V}\oplus {\mathcal V}^{\ast}$. By
Proposition~\ref{Courant-diffeo}, $\bar{\mathcal J}_{\varepsilon}$ is integrable if and only if ${\mathcal
J}_{\varepsilon}$ is so.

\smallskip

\noindent {\bf II}. Now for $a=(a',a'')\in A(E')\oplus A(E'')$ set
$$
a_1=pr_{TM}|E'\circ a'\circ (pr_{TM}|E')^{-1},\quad
a_2=pr_{TM}|E''\circ a''\circ (pr_{TM}|E'')^{-1}.
$$
Let $\varphi$ be the automorphism $a\to b=(b',b'')$ of $A(E')\oplus
A(E'')$ defined by
$$
b'=(pr_{TM}|E')^{-1}\circ a_2\circ pr_{TM}|E',\quad b''=(pr_{TM}|E'')^{-1}\circ a_1\circ pr_{TM}|E''.
$$
The differential $\varphi_{\ast}$ preserves the horizontal lifts.
Clearly, if $J\in {\mathcal G}$ gives rise to the complex structures
$(J_1,J_2)$ on $T_{\pi(J)}M$, then $\varphi(J)\in {\mathcal G}$  is
the generalized complex structure on $T_{\pi(J)}M$ determined by the
pair $(J_2,J_1)$. Moreover,  if $V\in{\mathcal V}_J$ gives rise to
the tangent vector $(V_1,V_2)$ of $Z(T_{\pi(J)}M,g)\times
Z(T_{\pi(J)}M,g)$ at $(J_1,J_2)$, then $\varphi_{\ast}V$ is the
vertical vector of ${\mathcal G}$ at $\varphi(J)$ determined by
$(V_2,V_1)$.

\smallskip

\noindent {\bf III}. Let $(M,\mathcal{J})$ and $(N,\mathcal{K})$ be two generalized
complex manifolds. Every diffeomorphism $\emph{f}: M\rightarrow N$
induces a bundle isomorphism
$F=\emph{f}_{\ast}\oplus\emph{f}^{\ast-1}:TM\oplus
T^{\ast}M\rightarrow TN\oplus T^{\ast}N$ and the identity
$F\circ\mathcal{J}=\mathcal{K}\circ F$ is a natural generalization
of the condition for a map between complex manifolds to be
holomorphic. The diffeomorphisms are not the only symmetries of
the generalized complex structures, the $B-$transforms are also
symmetries. Thus we  say that $(M,\mathcal{J})$ and
$(N,\mathcal{K})$ are equivalent if there is a diffeomorphism
$\emph{f}:\,M\rightarrow N$ and a closed $2$-form $B$ on $M$ such
that $F\circ e^{B}\mathcal{J}e^{-B}=\mathcal{K}\circ F$ (this is really an equivalence relation).
Since the form $B$ is closed, each of two equivalent generalized almost complex structures $\mathcal{J}$
and ${\mathcal K}$ is integrable if and only if the other one is so.

\smallskip
Let $\widehat{E}$ be the $B$-transform of $E$ by a $2$-form $\Psi$ on $M$. Then we have a natural diffeomorphism
$\beta$ of the generalized twistor spaces ${\mathcal G}={\mathcal G}(E)$ and
$\widehat{\mathcal G}={\mathcal G}(\widehat{E})$ sending a generalized complex structure $J\in{\mathcal G}$ to
its $B$-transform $\widehat{J}=e^{\Psi}Je^{-\Psi}$.

Denote by $D$ and $\widehat{D}$ the connections on $TM\oplus
T^{\ast}M$  determined by the generalized metrics $E$ and
$\widehat{E}$, respectively, as in Sec. \ref{main case}. Let
${\mathcal J}={\mathcal J}^{E}_{1}$ and $\widehat{\mathcal
J}={\mathcal J}^{\widehat{E}}_{1}$  be the generalized almost
complex structures on ${\mathcal G}$ and $\widehat{\mathcal G}$
defined by means of the connections  $D$ and $\widehat{D}$. If the
form $\Psi$ is closed, these generalized almost complex structures
are equivalent in a natural way. Indeed, set $E'=E$,
$\widehat{E}'=\widehat{E}$. The $B$-transform by $\Psi$ is an
orthogonal transformation of $TM\oplus T^{\ast}M$, thus it sends
$E''=E^{\perp}$ onto $\widehat{E}''=\widehat{E}^{\perp}$, the
orthogonal complements being taken with respect to the metric
$<.\,,.>$. Let $\nabla$ and $\widehat\nabla$ be the connections on
$TM$ obtained by transferring $D'=D|E$ and
$\widehat{D}'=\widehat{D}|\widehat{E}$. Recall that, on a Riemannian
manifold $(M,g)$, there is a unique metric connection with a given
torsion $T$ (for an explicit formula see, for example, \cite[Sec.
3.5, formula (14)]{GKM}). If  the torsion $3$-form
$T(X,Y,Z)=g(T(X,Y),Z)$
is skew-symmetric this connection can be written as \\
$\nabla^{LC}+\frac{1}{2}T$ where $\nabla^{LC}$ is the Levi-Civita
connection of $(M,g)$. Thus
$$
\begin{array}{c}
\nabla_{X}{Y}=\nabla^{LC}_{X}Y-\frac{1}{2}g^{-1}(\imath_{X}\imath_{Y}d\Theta),\\[6pt] \widehat{\nabla}_{X}{Y}
=\nabla^{LC}_{X}Y-\frac{1}{2}g^{-1}(\imath_{X}\imath_{Y}d\Theta)-\frac{1}{2}g^{-1}(\imath_{X}\imath_{Y}d\Psi).
\end{array}
$$
Hence
$$
\widehat{\nabla}_{X}{Y}=\nabla_{X}{Y}-\frac{1}{2}g^{-1}(\imath_{X}\imath_{Y}d\Psi).
$$
Suppose that the form $\Psi$ is closed, so that
$\widehat{\nabla}_{X}{Y}=\nabla_{X}{Y}$. Then the $B$-transform
$e^{\Psi}$ sends the connection $D$ to the connection $\widehat{D}$
since $pr_{TM}|E'=pr_{TM}|\widehat{E}'\circ e^{\Psi}$ and
$pr_{TM}|E''=pr_{TM}|\widehat{E}''\circ e^{\Psi}$.

It follows that  $\beta: L\to \widehat{L}=e^{\Psi}Le^{-\Psi}$ is an
isometry of  $A(E')\oplus A(E'')$ onto $A(\widehat{E}')\oplus
A(\widehat{E}'')$ sending the connection $D$ on  $A(E')\oplus
A(E'')$ induced by the connection $D|E'\oplus D|E''$  to the
connection $\widehat{D}'$ on $A(\widehat{E}')\oplus
A(\widehat{E}'')$ induced by $\widehat{D}|\widehat{E}'\oplus
\widehat{D}|\widehat{E}''$. In particular, $\beta_{\ast}$ preserves
the horizontal spaces,
\begin{equation}\label{B-Xh}
\beta_{\ast}X^h_L=X^{ \widehat h}_{\widehat{L}},\quad X\in TM,
\end{equation}
where $X^{\widehat h}$ is the horizontal lift of $X$  to
$T(A(\widehat{E}')\oplus A(\widehat{E}''))$.

The restriction of $\beta$ to ${\mathcal G}$ is a diffeomorphism of
${\mathcal G}$ onto $\widehat{\mathcal G}$ whose differential
preserves the horizontal spaces. Clearly,  $\beta_{\ast}$ preserves
also the vertical spaces sending a vertical vector $V$ at $J\in
{\mathcal G} $ to the vertical vector
$\widehat{V}=e^{\Psi}Ve^{-\Psi}$ at $\widehat{J}$. Then, if
$\alpha\in T^{\ast}_pM$, $Z\in T_pM$
$$
((\beta^{-1})^{\ast}\alpha^h _{J})(Z^{\widehat h}_{\widehat{J}})
=\alpha^h(\beta^{-1}_{\ast}Z^{\widehat
h}_{\widehat{J}})=\alpha^h_{J}(Z^h_{J})=\alpha(Z)
={\alpha}^{\widehat h}_{\widehat{J}}(Z^{\widehat h}_{\widehat{J}}),
$$
where $\alpha^{\widehat h}$ is the horizontal lift of $\alpha$ to
$T(A(\widehat{E}')\oplus A(\widehat{E}''))$. Also
$$
((\beta^{-1})^{\ast}\alpha^h
_{J})(\widehat{V})=\alpha^h(\beta^{-1}_{\ast}\widehat{V})=0
=\alpha^{\widehat h}_{\widehat{J}}(\widehat{V})
$$
for every vertical vector $\widehat{V}$ at $\widehat{J}$.
Thus
\begin{equation}\label{B-alphah}
(\beta^{-1})^{\ast}\alpha^h_J=\alpha^{\widehat h}_{\widehat J},\quad
\alpha\in T^{\ast}M.
\end{equation}
Note also that if $\Upsilon\in{\mathcal V}^{\ast}_{\widehat{J}}$,
$$
((\beta^{-1})^{\ast}\Upsilon)(\widehat{V})=\Upsilon(e^{-\Psi}\widehat{V}e^{\Psi}).
$$
Set
$$
{\mathscr B}=\beta_{\ast}\oplus(\beta^{-1})^{\ast},\quad
\widetilde{\Psi}=\pi^{\ast}\Psi,
$$
where, as before,  $\pi$ is the projection to $M$ of the bundle
$A(E')\oplus A(E'')$ restricted to ${\mathcal G}$. Taking into
account the fact that $B$-transforms act as the identity on
$1$-forms, we have
$$
{\mathscr B}(e^{\widetilde{\Psi}}{\mathcal
J}e^{-\widetilde{\Psi}}(\Upsilon))(\widehat{V}) ={\mathscr
B}(K_{1}^{\ast}\Upsilon)(\widehat{V})=(K_{1}^{\ast}\Upsilon)(e^{-\Psi}\widehat{V}e^{\Psi})
=\Upsilon(Je^{-\Psi}\widehat{V}e^{\Psi})
$$
and
$$
(\widehat{\mathcal J}{\mathscr B}(\Upsilon))(\widehat{V})={\mathscr
B}(\Upsilon)(\widehat{J}\widehat{V})=\Upsilon(e^{-\Psi}\widehat{J}\widehat{V}e^{\Psi})
=\Upsilon(Je^{-\Psi}\widehat{V}e^{\Psi}).
$$
Thus
$$
{\mathscr B}(e^{\widetilde{\Psi}}{\mathcal
J}e^{-\widetilde{\Psi}}(\Upsilon))=\widehat{\mathcal J}({\mathscr
B}(\Upsilon)).
$$
Also
$$
{\mathscr B}(e^{\widetilde{\Psi}}{\mathcal
J}e^{-\widetilde{\Psi}}(V))={\mathscr B}(JV)=e^{\Psi}JVe^{-\Psi}
=e^{\Psi}Je^{-\Psi}e^{\Psi}Ve^{-\Psi}=\widehat{J}{\mathscr
B}(V)=\widehat{\mathcal J}{\mathscr B}(V)
$$
since $\widetilde{\Psi}(V)=0$. For $J\in {\mathcal G}$, let
$(J_1,J_2)$ be the complex structures on $T_{\pi(J)}M$ determined by
$J$. Let $A=X+g(X)+\Theta(X)\in E'_{\pi(J)}$. Noting that
$\widetilde\Psi(X^h)=(\Psi(X))^h$, we have
$$
\begin{array}{l}
e^{\widetilde\Psi}{\mathcal J}e^{-\widetilde\Psi}(A^h_J)=e^{\widetilde\Psi}{\mathcal J}(X+g(X)+\Theta(X)-\Psi(X))^h_J\\[6pt]
=e^{\widetilde\Psi}(J_1X+g(J_1X)+\Theta(J_1X)-J\Psi(X))^h_J\\[6pt]
=[J_1X+g(J_1X)+\Theta(J_1X)-J\Psi(X)-\Psi(\pi_1(J\Psi(X)))]^h_J\\[6pt]
=[e^{\Psi}Je^{-\Psi}(A)]^h_{J}=(\widehat{J}(A))^h_J.
\end{array}
$$
Then,  by (\ref{B-Xh}) and (\ref{B-alphah}),
$$
{\mathscr B}(e^{\widetilde\Psi}{\mathcal
J}e^{-\widetilde\Psi}(A^h_J))=(\widehat{J}(A))^{\widehat
h}_{\widehat J}=\widehat{\mathcal J}(A^{\widehat h}_{\widehat
J})=\widehat{\mathcal J}({\mathscr B}(A^h_J)).
$$
Similarly, for $A=X-g(X)+\Theta(X)\in E''_{\pi(J)}$ in which case
$JA=J_2X-g(J_2X)+\Theta(J_2X)$.

This shows that ${\mathscr B}\circ (e^{\widetilde{\Psi}}{\mathcal J}e^{-\widetilde{\Psi}}) ={\widehat{\mathcal
J}}\circ{\mathscr B}$ where the $2$-form $\widetilde{\Psi}$ is closed.

A similar identity holds for another closed $2$-form $\bar{\Psi}$
under certain restrictions on the curvature of $M$. This form is
defined by
$$
\begin{array}{c}
\bar\Psi(X^h,Y^h)_J=\Psi(X,Y)_{\pi(J)}, \quad \bar\Psi(X^h,V)=\bar\Psi(V,X^h)=0,\\[6pt]
\bar\Psi(V,W)_J=G(V,K_1W),
\end{array}
$$
where $X,Y\in T_{\pi(J)}$, $V,W\in{\mathcal V}_J$. To prove the
identity ${\mathscr B}\circ (e^{\bar{\Psi}}{\mathcal
J}e^{-\bar{\Psi}})={\widehat{\mathcal J}}\circ{\mathscr B}$ we have
only to show that ${\mathscr B}(e^{\bar{\Psi}}{\mathcal
J}e^{-\bar{\Psi}}(V))={\widehat{\mathcal J}}{\mathscr B}(V)$. But
this follows  from the identity
$$
e^{\bar\Psi}{\mathcal J}e^{-\bar\Psi}(V)=K_1V-K_1^{\ast}(\bar\Psi(V))+\bar\Psi(K_1V)=
J\circ V-G(V)+G(V)=J\circ V.
$$
The standard formula for the differential in terms of the Lie
bracket and identity (\ref{[XhYh]}) imply
$d\bar{\Psi}(X^h,Y^h,Z^h)_J=0$. Let $\widetilde{a}$, $\widetilde{b}$
be the vertical vector fields obtained from sections $a$, $b$ of
$A(E')\oplus A(E'')$ such that $a(p)=V$, $b(p)=W$, $Da|_p=Db|_p=0$
for $p=\pi(J)$. Then, by Lemma~\ref{H-V brackets},
$d\bar{\Psi}(X^h,V,W)_J=X^h_JG(\widetilde{a},\widetilde{b})$ and it
is easy to see that $X^h_JG(\widetilde{a},\widetilde{b})=0$ using
formulas given in the proof of Lemma~\ref{H-V brackets}. Next,
$d\bar{\Psi}(X^h,Y^h,V)_J=0$ if and only if $G(R(X,Y)J,JV)=0$.
Therefore $d\bar{\Psi}=0$ if and only if $R(X,Y)J=0$ for every
$J\in{\mathcal G}$ and $X,Y\in T_{\pi(J)}M$. The latter condition is
equivalent to
\begin{equation}\label{last}
g({\mathcal R}(X\wedge Y),J_k Z\wedge U+Z\wedge J_k U)=0, \quad
k=1,2,\quad X,Y,Z,U\in T_{\pi(J)}M,
\end{equation}
where $(J_1,J_2)$ are the complex structures on $T_{\pi(J)}M$ determined by $J$.

Let $dim\,M=4$. In this case, identity (\ref{last}) for $J$ running over ${\mathcal G}_{++}$ (${\mathcal G}_{--}$)
is equivalent to $(M,g)$ being Ricci flat and anti-self-dual (self-dual, respectively). This identity holds on
${\mathcal G}_{+-}$ or ${\mathcal G}_{-+}$ if and only if $(M,g)$ is flat.

If $dim\,M\geq 6$, identity (\ref{last}) is equivalent to the flatness of $(M,g)$.

\smallskip

Finally, note that the complex structures on a tangent space of $M$ determined by $J\in{\mathcal G}$ and
$\widehat{J}=e^{\Psi}Je^{-\Psi}$ via (\ref{J1,2}) are the same. Therefore the diffeomorphism $\beta$ sends the connected
components ${\mathcal G}_{++},...,{\mathcal G}_{-+}$ of ${\mathcal G}$ onto the corresponding connected components
$\widehat{\mathcal G}_{++},...,\widehat{\mathcal G}_{-+}$ of $\widehat{\mathcal G}$. In the case $\Psi=-\Theta$ we have
$\widehat{E}=\{X+g(X): X\in TM\}$. Thus if $\Theta$ is closed  the integrability conditions for the generalized almost
complex structure ${\mathcal J}$ are the same as those for $\widehat{\mathcal J}$.

We summarize the considerations above as follows.

\begin{tm}
Let $E$ and $\widehat{E}$ be generalized metrics on a manifold $M$
determined by the same metric $g$ and $2$-forms $\Theta$ and
$\widehat\Theta$. If the $2$-form $\Theta-\widehat\Theta$ is closed,
the generalized almost complex structures ${\mathcal J}^{E}_{1}$ and
${\mathcal J}^{\widehat{E}}_{1}$  on the generalized twistor spaces
${\mathcal G}(E)$ and ${\mathcal G}(\widehat{E})$  are equivalent.
\end{tm}

\end{document}